%% file: ms.tex
\documentclass[11pt,a4paper]{article}
\usepackage{a4wide}
\usepackage[colorlinks=true,citecolor=blue]{hyperref}

\usepackage{graphicx}
\usepackage{float}
\usepackage{subcaption}
\usepackage{epstopdf}
\usepackage{floatflt} 
\usepackage{amsmath}
\usepackage{amsthm}
\usepackage{amssymb}
\usepackage{breqn} %Umbrüche in Matheumgebung
\usepackage{dsfont} % Indikatorfunktion
\usepackage{algorithmic}
\usepackage{algorithm}
\usepackage{xfrac}

\providecommand{\rhomax}[0]{\rho^{\max}}
\providecommand{\colvector}[1]{\left( \begin{array}{c} #1 \end{array} \right)}
\providecommand{\abs}[1]{\lvert#1\rvert}
\providecommand{\vmax}[0]{v^{\max}}

\DeclareMathOperator*{\ADPARZ}{AP}
\DeclareMathOperator*{\nice12}{\sfrac{1}{2}}

\usepackage{tikz}
\newlength\fwidth
\usetikzlibrary{arrows}
\usetikzlibrary{decorations.pathreplacing}
\usetikzlibrary{plotmarks}
\usepackage{pgfplots} 
\usepackage{pgfgantt}
\usepackage{pdflscape}
\pgfplotsset{compat=newest} 
\pgfplotsset{plot coordinates/math parser=false}
\usepackage{varwidth}
\usepgfplotslibrary{fillbetween}
\pgfplotsset{compat = 1.3}
\usepackage{booktabs}
\usepackage{multirow}

\usepackage{color}

\newtheorem{definition}[]{Definition}[]
\newtheorem{proposition}[]{Proposition}[]
\newtheorem{lemma}[]{Lemma}[]

\begin{document}
 
\title{Second-order traffic flow models  on networks }

\author{Simone G\"ottlich\footnotemark[1]  \; Michael Herty\footnotemark[2] \; Salissou Moutari\footnotemark[3] \; Jennifer Wei{\ss}en\footnotemark[1]}

\footnotetext[1]{University of Mannheim, Department of Mathematics, 68131 Mannheim, Germany (goettlich@uni-mannheim.de, jennifer.weissen@uni-mannheim.de)}
\footnotetext[2]{RWTH Aachen University, IGPM, 52064 Aachen, Germany (herty@igpm.rwth-aachen.de)}
\footnotetext[3]{Queen's University Belfast, BT7 1NN Belfast, United Kingdom (s.moutari@qub.ac.uk)}

\maketitle

\begin{abstract}
	This paper deals with the Aw-Rascle-Zhang model for traffic flow on uni-directional road networks. For the conservation of the mass and the generalized momentum, we construct weak solutions for Riemann problems at the junctions. 
	We particularly focus on a novel approximation to the homogenized pressure by introducing an additional equation for the propagation of a reference pressure. The resulting system of coupled conservation laws is then solved using an 
	appropriate numerical scheme of Godunov type. Numerical simulations show that the proposed approximation is able to approximate the  homogenized pressure sufficiently well. The difference of the new approach compared with the Lighthill-Whitham-Richards model is also illustrated.
\end{abstract}

% Conservation laws, Traffic problems, Finite volume methods


%%%%%%%%%%%%%%%%%%%%%%%%%%%%%%%%%%%%%%%%%%%%

	\section{Introduction}
	
	Macroscopic modeling of traffic flow has been of research interest over the past decades, since the first order Lighthill-Whitham-Richards (LWR) model \cite{LigWhi1955,Ric1956} 
	\begin{align}\label{lwr}
	\partial_t \rho + \partial_x ( \rho V(\rho) )&= 0
	\end{align}
	was introduced in the middle of the 20th century. Aw-Rascle-Zhang~\cite{AwRas2000,Zha2002} proposed a second order traffic flow model (ARZ model), which accounts for a more detailed description of traffic phenomena. 
	The ARZ model is given by
	\begin{equation}
	\begin{split}\label{eq:ARZ}
	\partial_t \rho + \partial_x (\rho v) &= 0 \\
	\partial_t (\rho w) + \partial_x (\rho v w) &=0. 
	\end{split}
	\end{equation}
	It combines a conservation law for the density $\rho$ with an additional conservation law for the mean traffic speed  $v$. The Lagrangian marker $w$ is given by 
	\begin{align}\label{eq:pressure} 
	w = v + p(\rho),
	\end{align}  where $p$ is a pressure function. The LWR model can be viewed as a special case of the second order model, where the speed of traffic is always at the level of the equilibrium speed $V(\rho)$ and all the drivers possess the {\bf  same} Lagrangian marker $w = V(\rho) + p(\rho) =w_0$. 
	In comparison to the LWR model, which is in general not able to capture traffic instabilities, the ARZ model enables to portray traffic phenomena such as growing traffic waves and the capacity drop effect at junctions~\cite{TreKes2013}.
	
	Coupling or boundary conditions are required to extend equations \ref{lwr} or \eqref{eq:ARZ} to models for traffic flow on  road networks. The appropriate junction modelling  is of utmost importance to model traffic dynamics and it has been the focus of recent research in the field, see e.g.  \cite{BreCanGar2014,GarHanPic2016, trafficflowonnetworks}.  The proposed  models for traffic are generally based on the conservation of mass at the junctions. At the same time, further conditions are imposed to obtain a unique boundary condition. As an example, the non--negative flux $\rho v$ at the junction is distributed according to given ratios and the total flux through the junction is maximized according to the specified ratios.  While many studies, including ~\cite{CocGarPic2005,GarHanPic2016,trafficflowonnetworks,GoaGoeKol2016, HerKla2003, HolRis1995}, use the LWR model \ref{lwr} to investigate traffic modelling on road networks, we are interested in models based on the ARZ model \eqref{eq:ARZ}. Further, we 
	are interested in a description in terms of the so-called  macroscopic quantities $(\rho, \rho v).$ Since macroscopic models are known to be coarse in the description of traffic dynamics, especially at junctions,  the maximization of the flux at junctions may sometimes lead to unrealistically high values of flux $\rho v$.  To address these shortcomings, hybrid models, providing a more accurate description of traffic dynamics at the junctions, have been introduced~\cite{MR3886680,MouHer2009, WeiKolGoe2019}. These models are computationally more expensive and we refer to the cited literature for a more detailed discussion. 
	
	Here, we are interested in junction models for the ARZ model. Conditions have been considered for example in~\cite{BuliXing2020,GarHanPic2016,trafficflowonnetworks, HauBas2007, HerMouRas2006, HerRas2006, KolCosGoa2018, KolGoeGoa2017}.  A detailed review  would be beyond the scope of the paper, however, note that the Riemann solver at the junction proposed in~\cite{GarPic2006} does not conserve the pseudo momentum and therefore does not yield a weak solution to the ARZ model~\eqref{eq:ARZ}. Except for \cite{HerMouRas2006,HerRas2006},  {\bf all} proposed coupling conditions ~\cite{HauBas2007, KolCosGoa2018, KolGoeGoa2017} use  a fixed pressure function $p=p(\rho)$, see equation~\eqref{eq:pressure} on each road of the network.  The coupling conditions therein conserve both the mass and a (predefined) mixture of the Lagrangian attribute $w$. Specifically, the case of a 2-1 merging junction, where the Riemann solver is based on fixed assigned coefficients, was presented~\cite{KolCosGoa2018}. The flux through the junction was maximized through a multi-objective maximization of the incoming fluxes.  Other examples \cite{GarPic2006} also propose conversation of mass and a mixture of the Lagrangian marker. However, as outlined in \cite{HerRas2006}, those solutions
	are {\bf not} consistent with the Lagrangian form of equation \eqref{eq:ARZ}. Without entering the discussion in detail, the crucial point is that any change in the
	value of $w$ induces also a change in the corresponding pressure $p$ according to equation \eqref{eq:pressure}. This leads to coupling conditions at junctions 
	that not only prescribe coupling or boundary conditions in terms of $(\rho,\rho v)$ but {\bf also} in the form of the pressure $p,$ see e.g.~\cite{HerMouRas2006,HerRas2006}. This interpretation is also consistent with a discretization in Lagrangian coordinates leading to the so--called follow-the-leader
	model \cite{AwKlaMat2002}.  In order to pass from a microscopic to a macroscopic limit, only low regularity is required on the Lagrangian marker $w$. The 
	homogenization limit of the system \eqref{eq:ARZ} was investigated in~\cite{BagRas2003}, where the homogenized relationship between the traffic density, velocity and Lagrangian marker through the pressure function was established. In~\cite{HerMouRas2006, HerRas2006}, coupling conditions for the pressure using the  homogenized pressure have been established. Therein, the derivation of the homogenized pressure for the 2-1 merge junction and the situation where the initial pressure $p(\rho)$ is given by $p(\rho) = \rho$ has been illustrated. It should also be noted that whenever the incoming $w$-values or the mixture rule changes, the homogenized pressure  changes. This leads to computational challenges and to the best of our knowledge, an explicit formula to determine the adapted pressure and a numerical scheme complying with these changes are yet to be established. In this paper we propose a numerical approximation to the homogenized pressure in order to obtain a computable but still consistent traffic flow model based on the ARZ equations \eqref{eq:ARZ}.
	
	This paper focuses on the following three important aspects. First and in contrast with~\cite{GarPic2006}, the solution proposed in this study is a weak solution of the network problem and therefore satisfies the conservation of mass and generalized momentum. Thus, the quantities $\rho v$ and $\rho w v$ are conserved through the junction. Second, compared to ~\cite{HauBas2007,KolCosGoa2018,KolGoeGoa2017}, this study considers the homogenization of the pressure at the junction and, hence   it   is consistent with the formulation of the model in Lagrangian coordinates and to the microscopic follow-the-leader model.  Finally, in contrast   ~\cite{HerMouRas2006,HerRas2006}, we present an approach for approximating the homogenized pressure.  Such an approach allows for numerical simulation of traffic in an entire road network and the corresponding computational results are presented. 
	
	After the presentation of the approximation of the homogenized pressure, we discuss the properties of the homogenized system and construct the demand and supply functions. These are used to determine the flux at the junction and allow for an appropriate description of the boundary conditions. Using the approximation of the homogenized pressure in the supply function, admissible states at the junction are defined and the network solution is constructed. We provide a suitable numerical scheme for simulating traffic on road networks, which is easy to implement and based on a Godunov discretization. It is well known, that the classical Godunov scheme produces nonphysical oscillations near contact discontinuities~\cite{ChaGoa2007}. These oscillations then lead to numerical solutions that do not precisely capture the Riemann invariants of the system~\eqref{eq:ARZ}. Since an accurate description of the Riemann invariants is of paramount importance for our network model, we leverage the non-conservative scheme from~\cite{ChaGoa2007}, which was specifically designed for the system~\eqref{eq:ARZ}, and adapt it to our network model. The numerical examples show that the solution is sufficiently close to the homogenized solution.
	
	% The outline is not required, but we show an example here.
	The outline of the paper is as follows: In Section~\ref{sec:overview_coupling}, we exemplify some coupling conditions for the LWR and the AR models, and we define weak entropy solutions for the corresponding Riemann problem. As in \cite{HerRas2006}, the need for the homogenization of the pressure function after merge type junctions is shown. In Section~\ref{sec:pressurederivation}, we provide a simple and suitable approximation for the adapted pressure function, and we show that with our approach, we can indeed approximate the homogenized pressure and the correct homogenized solution. The approximation is then generalized for the general $n$-$m$-junction and the quality of the solution is assessed  through a comparison of the flux with true and approximated homogenized pressure for the 2-1-junction. A numerical scheme, which accounts for the variation of the pressure, is introduced in Section~\ref{sec:numscheme}. Finally, we  compare our network solution against the solution with the homogenized pressure in Section~\ref{sec:computationalresults}. Furthermore, the numerical results highlight a more accurate description of traffic dynamics using the proposed approach compared to the LWR network model.
	
	%%%%%%%%%%%%%%%%%%%%%%%%%%%%%%%%%%%%%%%%%%%%%%%%%%%%%%%%%%%%%%%%%%%%%%%%%%%%%%%%%%%%%%%%%%%%%%%%%%%%%%%%%%%%%%%%%%%%%%%%%%%%%%%%%%%%%%%%%%%%%%%%%%%%%%%%%%%%%%%%%%%%%%%%%%%%%%%%%%%%%%%%%%%%%%%%%%%%%%%%%%%%%%%%%%%%%%%%%%%%%%%%%%%%%%%%%%%%%%%%%%%%%%%%%%%%%%%%%%%%%%%%%%%%%%%%%%%%%%%%%%%%%%%%%%%%%%%%%%%%%%%%%%%%%%%%%%%%%%%%%%%%%%%%%%%%%%%%%%%%%%%%%%%%%%%%%%%%%%%%%%%%%%%%%%%%%%%%%%%%%%%%%%%%%%%%%%%%%%%%%%%%%%%%%%%%%%%%%%%%%%%%%%%%%%%%%%%%%%%%%%%%%%%%%%%%%%%%%%%%%%%%%%%%%%%%%%%%%%%%%%%%%%%%%%%%%%%%%%%%%%%%%%%%%%%%%%%%%%%%%%%%%%%%%%%%%%%%%%%%%%%%%%%%%%%%%%%%%%%%%%%%%%%%%%%%%%%%%%%%%%%%%%%%%%%%%%%%%%%%%%%%%%%%%%%%%%%%%%%%%%%%%%%%%%%%%%%%%%%%%%%%%%%%%%%%%%%%%%%%%%%%%%%%%%%%%%%%%%%%%%%%%%%%%%%%%%%%%%%%%%%%%%%%%%%%%%%%%%%%%%%%%%%%%%%%
	
	\section{Coupling conditions for traffic flow networks} \label{sec:overview_coupling}

	A road network is modeled as a directed graph $\mathcal{G} = (V, E)$. Each edge  $i \in E$ corresponds to a road, which is modeled as an interval $I_i = [a_i, b_i]$ with length $L_i = b_i -a_i$. The junctions are represented by the nodes $k \in V$. For a given junction $k$, let $\delta_{k}^-$ (resp. $\delta_{k}^+$) denote the set of indices representing incoming (resp. outgoing) roads to (resp. from) the junction $k$. On each road $i \in E$ of the network, a traffic network model given by \ref{lwr} or \eqref{eq:ARZ} is required to hold. Furthermore, some initial data are assumed to be described on each road. Coupling conditions at the junctions are imposed to define suitable boundary conditions for the traffic model at hand. Those conditions will be discussed in detail in the following section. 
	
	\subsection{The Lighthill-Whitham-Richards model}
	On each road $i \in E$ of the network, we require the following equation to hold:
	\begin{align}
	\partial_t \rho_i  + \partial_x \left( \rho_i V_i(\rho_i) \right)  = 0. \label{eq:LWR_network}
	\end{align}
	The traffic density and velocity on road $i$ are denoted $\rho_i = \rho_i(x,t)$ and $v_i = V(\rho_i(x,t))$, respectively. The velocity is a function of the density and for a maximum density $\rhomax_i$ it holds that $V_i(\rhomax_i) = 0$. Moreover, the flux $\rho V_i(\rho)$ is strictly concave and has a unique maximum $\sigma_i$. For a given junction $k$, let $ \lbrace \Phi_i \rbrace_{i \in (\delta_{k}^- \cup \delta_{k}^+)}$ denote a family of smooth test functions, where $\Phi_i: I_i \times [0, + \infty] \rightarrow \mathbb{R}^2 $ has a compact support in $I_i$ and is also smooth across the junction, i.e., 
	\begin{align*}
	\Phi_i( b_i, \cdot ) = \Phi_j( a_j, \cdot ) \qquad \forall i \in \delta_{k}^-, \forall j \in \delta_{k}^+.
	\end{align*}
	A set of functions $\rho_i, i \in (\delta_{k}^- \cup \delta_{k}^+)$,  is called a weak solution of~\eqref{eq:LWR_network} at the junction $k$ if, for all families of test functions smooth across the junction, the following equation holds:
	\begin{align} \label{eq:networkproblem1_LWR}
	\sum_{i \in (\delta_{k}^- \cup \delta_{k}^+)} \left(\int_{0}^{\infty} \int_{a_i}^{b_i} \left[\rho_i  \cdot  \partial_t \Phi_i + \rho_i v_i  \cdot \partial_x \Phi_i \right] \mathrm{d}x \mathrm{d }t- \int_{a_i}^{b_i} \rho_{i,0}  \cdot \Phi_i(0,x) \mathrm{d }x \right) = 0.
	\end{align}
	In the above equation, $\rho_{i,0}$ denotes the initial data. First, we provide a discussion of the Riemann problem at a single junction $k$, located at $x=0$. For each road $i \in (\delta_{k}^+ \cup \delta_{k}^-)$, we consider the following (half-)Riemann problem \cite{HerRas2006}: 
	\begin{align} \label{eq:LWR_RP_network}
	\begin{cases}
	\partial_t \rho_i  + \partial_x  ( \rho_i v_i )  = 0 \\
	\rho_i(x,0) = \begin{pmatrix} \rho_i^+ & \mbox{ for } x>0 \\ \rho_i^- & \mbox{ for } x \leq 0. \end{pmatrix}. 	
	\end{cases}
	\end{align}
	Depending on whether the road is incoming or outgoing, only one of the Riemann data is defined for $t=0$. If $i \in \delta_{k}^-$, then $\rho_i^- = \rho_{i,0}, b_i = 0$ and if $i \in \delta_{k}^+$, then $\rho_i^+ = \rho_{i,0},$ which will be  the case $ a_i = 0$. The other datum is defined by the solution through some suitable coupling conditions  be discussed below.

	\subsubsection{Nodal conditions for the LWR model}
	
	In this section, we discuss the construction of weak entropy solutions for the Riemann problem~\eqref{eq:LWR_RP_network} for traffic networks. On each road, we consider the LWR model. To define the solution at the junction, we consider a Riemann solver giving an admissible, weak entropy solution, see also~\cite{CocGarPic2005, GarHanPic2016, trafficflowonnetworks, HolRis1995}. The entropy criterion is expressed by a demand and supply formulation for admissible flux values. We impose additional conditions on the flux distribution from incoming to outgoing roads in the network. Using a flux maximization with respect to the additional conditions, we obtain the network solution. In the following, we specify the constraints on the network solution. From~\eqref{eq:networkproblem1_LWR}, a weak solution satisfies the Kirchhoff condition 
	\begin{align} \label{eq:RankineHugoniot_mass}
	\sum_{i \in \delta_{k}^-} \underbrace{(\rho_i v_i)(0-,t)}_{:=q_i} = \sum_{j \in \delta_{k}^+} \underbrace{(\rho_j v_j)(0+,t)}_{:=q_j}.
	\end{align}
	Let $q_{ji} \in \mathbb{R}_{\geq 0}$  for $j \in \delta_{k}^+, i \in \delta_{k}^-$ denote the initially unknown flux coming from road $i$ and entering road $j$ and denote by $q_i$ and $q_j$ the total fluxes at the junction: 
	\begin{align*}
	q_i = \sum_{j \in \delta_{k}^+} q_{ji}, \qquad& q_j = \sum_{i \in \delta_{k}^-} q_{ji}.
	\end{align*}
	
	The following constraint (H1) specifies the admissible fluxes for entropy solutions. Conditions for the flux distributions between different roads of the network will be given in (H2)-(H3) below.

	\textbf{(H1)}
	The fluxes at the junction are bounded by demand $d_i$ and supply $s_j$
	\begin{align}
	0 \leq q_i \leq d_i(p_i) \quad \forall i \in \delta_{k}^-, \qquad 0 \leq q_j \leq s_j(\rho_j) \quad \forall j \in \delta_{k}^+,
	\end{align}
	where the demand, $d_i(\rho)$, and the supply, $s_i(\rho)$, on road $i$ are defined as follows:
	\begin{align} \label{eq:LWR_demandsupply}
	d_i(\rho) =\begin{cases}
	\rho V_i(\rho) &\text{ if } \rho \leq \sigma_i \\
	\sigma_i V_i(\sigma_i)  &\text{ if } \rho > \sigma_i 
	\end{cases}, \qquad
	s_i(\rho) =\begin{cases}
	\sigma_i V_i(\sigma_i) &\text{ if } \rho \leq \sigma_i \\
	\rho V_i(\rho)  &\text{ if } \rho > \sigma_i 
	\end{cases}
	\end{align}

	\textbf{(H2)} Consider a junction with $n$ incoming and $m$ outgoing roads. As in \cite{CocGarPic2005,GarHanPic2016,trafficflowonnetworks,HerRas2006,HolRis1995}, a traffic distribution matrix $A = (\alpha_{ji})_{i \in \delta_{k}^-, j \in \delta_{k}^+}$ is assumed to be known. It describes the distribution of traffic at the junction,  where $0 \leq \alpha_{ji} \leq 1$ denotes the percentage of cars on road $i$ willing to go to road $j$ and $\sum_{j \in \delta_{k}^+} \alpha_{ji} = 1$. The fluxes at the junction must  satisfy the following equality  $q_{ji} = \alpha_{ji} q_i.$

	For junctions with more than one incoming road, conditions (H1)-(H2) are not sufficient to determine unique flux values at the junction. A further constraint has to be introduced to single out a solution. Here, we impose the following mixture rule, see also~\cite{trafficflowonnetworks,HerRas2006}.

	\textbf{(H3)}
	On an outgoing road $j$, the fluxes $\vec{q_i} = ({q_i})_{i \in \delta_{k}^-} $ are proportional to a given priority vector $\vec{\beta} =(\beta_{ij})_{i \in \delta_{k}^-}$. Hence, we impose the following constraint:
	\begin{align}\label{eq:mixturerule_junction}
	\vec{q_i} = z \vec{\beta}, \qquad z \in \mathbb{R}.
	\end{align}
	
	%With the constraints above, we proceed to define the network solution in~\Cref{def:LWR_networksolution}. 
	We require the network solution to be  the flux maximizing weak solution, subject to constraints (H1)-(H3). 
	The existence of the network solution 
	%according to~\Cref{def:LWR_networksolution} 
	is shown in the references~\cite{trafficflowonnetworks, HolRis1995}.

	\subsection{The Aw--Rascle--Zhang model}
	
	On each road of the network, we require the following system to hold:
	\begin{align}
	\partial_t \colvector{\rho_i \\ \rho_i w_i} + \partial_x \colvector{\rho_i v_i \\ \rho_i w_i v_i} = 0, \label{eq:ARZ_network}
	\end{align}
	where $\rho_i$  and $v_i$ are the density and the velocity on road $i$, respectively. The Lagrangian marker, $w_i$, is defined by $w_i = v_i + p_i(\rho_i).$ The form of the pressure function $p_i$ depends on the initial data and the type of the junction. For each $i$, $p_{i,0}(\rho)$ is a pressure function, which is initially given and for which the flux $\rho (w-p_{i,0}(\rho))$ has a unique maximum $\sigma_{i,0}$.  A set of functions $U_i = (\rho_i, w_i), i \in (\delta_{k}^- \cup \delta_{k}^+)$ is called a weak solution of~\eqref{eq:ARZ} at the junction $k$ if, for all families of test functions smooth across the junction, the following equation holds:

	\begin{align} 
	\begin{split} \label{eq:networkproblem1}
	\sum_{i \in (\delta_{k}^- \cup \delta_{k}^+)} \Bigg ( \int_{0}^{\infty} \int_{a_i}^{b_i} \left[\colvector{\rho_i \\ \rho_i w_i} \cdot  \partial_t \Phi_i + \colvector{\rho_i v_i \\ \rho_i v_i w_i} \cdot \partial_x \Phi_i \right] \mathrm{d}x \mathrm{d }t \\
	- \int_{a_i}^{b_i} \colvector{\rho_{i,0} \\ \rho_{i,0} w_{i,0}} \cdot \Phi_i(0,x) \mathrm{d }x \Bigg ) = 0.
	\end{split}
	\end{align}
	
	We refer to \cite{AwRas2000} for an analysis of the model \eqref{eq:ARZ_network}. As before, for each $i \in (\delta_{k}^+ \cup \delta_{k}^-)$, we consider the following (half-)Riemann problem:
	\begin{align} \label{eq:ARZ_RP_network}
	\begin{cases}
	\partial_t \colvector{\rho_i \\ \rho_i w_i } + \partial_x \colvector{ \rho_i v_i \\ \rho_i w_i v_i } = 0, \\
	(\rho_i, w_i) = \begin{pmatrix} (\rho_i^+,w_i^+) & \mbox{ for } x>0 \\ (\rho_i^-, w_i^-) & \mbox{ for } x \leq 0. \end{pmatrix}. \\
	\end{cases}
	\end{align}
	As in the case of the LWR model, depending on whether the road is incoming or outgoing, only one of the Riemann data is defined for $t=0$. In the following, we denote with $U = (\rho, \rho w)$ the traffic state defined by the density $\rho$ and the Lagrangian marker $w$.
	
	\subsubsection{Nodal conditions for the ARZ model}
	
	In this section we state conditions for a network solution to the  ARZ model.  To define a unique network solution for the Riemann problem~\eqref{eq:ARZ_RP_network} of the $2 \times 2$ system~\eqref{eq:ARZ}, additional conditions compared 
	to the Riemann problem of the scalar conservation law~\eqref{eq:LWR_RP_network},  are necessary, see also~\cite{HauBas2007,HerMouRas2006, HerRas2006, KolCosGoa2018}. From~\eqref{eq:networkproblem1}, a weak solution for the ARZ model is required to satisfy the Kirchhoff condition for the conservation of mass~\eqref{eq:RankineHugoniot_mass} as well as momentum: 
	\begin{align}
	\sum_{i \in \delta_{k}^-} (\rho_i v_i w_i)(0-,t) = \sum_{j \in \delta_{k}^+} (\rho_j v_j w_j)(0+,t). \label{eq:conservation_pseudomomentum}
	\end{align}
	
	Additionally, the homogenization of the Lagrangian marker given by the rule (H4) below has to be considered~\cite{HerMouRas2006,HerRas2006}.

	\textbf{(H4)}
	The homogenized $w$-value on an outgoing road $j$ is motivated by the underlying microscopic model, see section 6 in~\cite{HerRas2006}. Cars passing trough the junction have the average property $\overline{w_j}$ associated with the Young measure $\mu_x$ describing the mixture of cars. Once the proportions $\beta_{ij} = q_{ji}/q_j$ are known, the homogenized Lagrangian marker $\overline{w_j}$ is given by
	\begin{align}\label{eq:barw_j}
	\overline{w_j} &:= \int w \mathrm{d }\mu_x(w) = \sum_{i \in \delta_{v}^-} \frac{q_{ji}}{q_j} w_i(U_{i,0}) = \vec{\beta} {\vec{w}}^T ~~\forall j \in \delta_{k}^+.
	\end{align}
	The unique weak entropy solution of~\eqref{eq:ARZ_network} is characterized by the relation
	\begin{align}
	\tau = \sum_{i \in \delta_{k}^-} \frac{q_{ji}}{q_j} P_j^{-1}(w_{i,0} -v) =: (P_j^*)^{-1}(\overline{w_j}-v), \label{eq:adaptedpress1}
	\end{align}
	see~\cite{HerRas2006}. The pressure $P_j^*(\tau)$ is, in fact, redefined such that for each $\tau$, the velocity $v$ is the velocity of the weak entropy solution of the Lagrangian formulation of the ARZ equations. This solution is then equivalent to the solution of the ARZ equations ~\eqref{eq:ARZ_network}. The modified pressure function and the Lagrangian marker are defined as follows:
	\begin{align} 
	p_j^*(\rho) = P_j^*(1/\rho), \label{eq:adaptedpress2} \\
	w_j(U) = \overline{w_j} =  v +p^{*}_j(\rho). \label{eq:adaptedpress3}
	\end{align}
	The new pressure function corresponds to  the homogenized Lagrangian marker $\overline{w_j}$, and for different values $w \neq \overline{w_j}$, a different pressure $p^*$ is obtained. The construction of the new pressure function is well-defined once the proportions $\frac{q_{ji}}{q_j}$ of the incoming fluxes are known. 
	The total fluxes $q_i$ and $q_j$ fulfill similar bounds in terms of  the demand and supply as in the case of the LWR model.  More precisely, we replace (H1)
	by
	
	\textbf{(H1*)}
	The fluxes at the junction are bounded by the demand $d_i$ and the supply~ $s_j$
	\begin{align} \label{eq:network_demandandsupplyjunction}
	0 \leq q_i \leq d_i(p_i;\rho_{i,0},w_{i,0}) \quad \forall i \in \delta_{k}^- \qquad 0 \leq q_j \leq s_j( p^{*}_j; \tilde{\rho_j}, \overline{w_j}) \quad \forall j \in \delta_{k}^+.
	\end{align}
	Here, the density  $\tilde{\rho_j}$ are defined below in Section \ref{sec:numscheme}. The value $\overline{w_j}$ is the homogenized Lagrangian marker
	given by equation \eqref{eq:barw_j}. Furthermore, the pressure $p^*_j$ is the homogenized pressure given by equation \eqref{eq:adaptedpress2}.
	
	The demand and supply functions for the ARZ model are given by
	\begin{align*}
	&
	d_i(p;\rho, w) =\begin{cases}
	\rho (w - p(\rho)) &\text{ if } \rho \leq \sigma_i(w) \\
	\sigma_i(w)(w- p(\sigma_i(w)))  &\text{ if } \rho > \sigma_i(w) 
	\end{cases},
	\\
	& 
	s_i(p;\rho, v) =\begin{cases}
	\sigma_i(w) (w - p(\sigma_i(w))) &\text{~~if } \rho \leq \sigma_i(w) \\
	\rho (w- p(\rho))  &\text{~~if } \rho > \sigma_i(w)
	\end{cases},
	\end{align*}
	where $\sigma_i(w)$ is the unique maximum of the flux $\rho (w-p_i(\rho))$. For a given state $U =(\rho, \rho w)$, the demand on the incoming road is evaluated using $w(U) = v +p_i(\rho)$ whereas the supply on the outgoing road uses $w(U) = v + p^*_j(\rho)$. 
	
	We refer to the reference~\cite{HerRas2006} (section 7, Theorem 7.1) for the proof of existence and uniqueness of the network solution. % according to~\Cref{def:networksolution_ARZ}.
	
	%%%%%%%%%%%%%%%%%%%%%%%%%%%%%%%%%%%%%%%%%%%%%%%%%%%%%%%%%%%%%%%%%%%%%%%%%%%%%%%%%%%%%%%%%%%%%%%%%%%%%%%%%%%%%%%%%%%%%%%%%%%%%%%%%%%%%%%%%%%%%%%%%%%%%%%%%%%%%%%%%%%%%%%%%%%%%%%%%%%%%%%%%%%%%%%%%%%%%%%%%%%%%%%%%%%%%%%%%%%%%%%%%%%%%%%%%%%%%%%%%%%%%%%%%%%%%%%%%%%%%%%%%%%%%%%%%%%%%%%%%%%%%%%%%%%%%%%%%%%%%%%%%%%%%%%%%%%%%%%%%%%%%%%%%%%%%%%%%%%%%%%%%%%%%%%%%%%%%%%%%%%%%%%%%%%%%%%%%%%%%%%%%%%%%%%%%%%%%%%%%%%%%%%%%%%%%%%%%%%%%%%%%%%%%%%%%%%%%%%%%%%%%%%%%%%%%%%%%%%%%%%%%%%%%%%%%%%%%%%%%%%%%%%%%%%%%%%%%%%%%%%%%%%%%%%%%%%%%%%%%%%%%%%%%%%%%%%%%%%%%%%%%%%%%%%%%%%%%%%%%%%%%%%%%%%%%%%%%%%%%%%%%%%%%%%%%%%%%%%%%%%%%%%%%%%%%%%%%%%%%%%%%%%%%%%%%%%%%%%%%%%%%%%%%%%%%%%%%%%%%%%%%%%%%%%%%%%%%%%%%%%%%%%%%%%%%%%%%%%%%%%%%%%%%%%%%%%%%%%%%%%%%%%%%%%%
	
	\section{Approximation of homogenized pressure} \label{sec:pressurederivation}
	
	Suitable conditions for finding the network solution %according to Definition~\cref{def:networksolution_ARZ} 
	have been proposed in \cite{HerRas2006}, where a full discussion of the Riemann problem was presented. It is assumed that the proportions $q_{ji}/q_j$ of the incoming fluxes at the junction are known and fixed. The homogenized flow on parts of the outgoing roads is derived. However, even in the case where initially Riemann data  are available on all roads, the pressure law on the outgoing roads changes over time. Further,  the adapted pressure function defined by~\eqref{eq:adaptedpress2} is not obvious to find, due to the nonlinear and implicit relationship between \eqref{eq:barw_j} and \eqref{eq:adaptedpress1}. The coupling conditions suggested by \cite{HerMouRas2006,HerRas2006} require the adapted pressure function to be recalculated for any  change in the homogenized Lagrangian marker $\overline{w} = \vec{\beta} \vec{w}^T$. Due to this complexity, the adaption of the pressure law in numerical simulations for traffic networks has not been considered so far. In this work, we suggest a practical approach for the simulation of traffic networks, which accounts for the changes in the pressure law on outgoing roads. We present an approximation for computing the adapted pressure, whenever the homogenized Lagrangian marker at the junctions changes. To achieve this, we replace (H4) by the following condition on the pressure:
	\\
	\\
	\textbf{(H4*)}
	Assume that the pressure on an outgoing road is initially given by $p(\rho)$, the Lagrangian markers on the roads entering the junction are given by $\vec{w}$ and the priority vector is $\vec{\beta}$. The homogenized Lagrangian marker is given by~\eqref{eq:barw_j} and we approximate the homogenized pressure~\eqref{eq:adaptedpress2} as follows:
	\begin{align} \label{eq:pressureapproximation}
	p^{*}(\rho) \approx p^{**}(\rho) = c(\vec{\beta}, \vec{w})  p(\rho).
	\end{align}
	
	The value of $c(\vec{\beta}, \vec{w})$ will be defined below. Its value will be transported with the outgoing velocity. Approximation properties and further details will be discussed in the following sections. 
	
	\subsection{Properties of the homogenized system} \label{sec:homogenizedsystem}
	
	Assuming the pressure law is of the type $c(\vec{\beta}, \vec{w}) p(\rho)$ with a value $c(\vec{\beta}, \vec{w})$ independent of $\rho,$  the pressure propagates with the velocity $v$ of the cars, according to equation \eqref{eq:ARZ_network}.  However, the value of $c$ might change over time   due to the coupling at the junction. This will change the value of $c$ corresponding to the value of the homogenized Lagrangian marker $\overline{w}.$ Once this marker enters the outgoing road it is transported with the flow. 
	The idea is to propose an additional advection equation to the original system \eqref{eq:ARZ} to propagate the value of $c(\vec{\beta}, \vec{w})$ together with the marker $\overline{w}.$ 
	The value of $c(\vec{\beta}, \vec{w})$ is then used as correction to the pressure $p(\rho).$ We will show below that this correction $ c(\vec{\beta},\vec{w}) p (\rho)$ is in fact an approximation to 
	the homogenized pressure in the sense of equation \eqref{eq:pressureapproximation}. 
	\par 
	If $c$ is transported similarly to $w$, then in conservative form, the system reads:  
	\begin{align}
	\partial_t \colvector{\rho \\ \rho w \\ \rho c} + \partial_x \colvector{ \rho v \\ \rho w v \\ \rho c v} = 0. \label{eq:ADPAR_Eulerian}
	\end{align}
	This system (with a different motivation) has been studied in ~\cite{BagRas2003}. In contrast with~\cite{BagRas2003}, here we do not work with the Lagrangian formulation, but instead consider the Eulerian formulation. The meaning of $c$ is also different compared with \cite{BagRas2003}. Nevertheless the mathematical properties are the same and will be recalled here for convenience.  The system~\eqref{eq:ADPAR_Eulerian} is hyperbolic with eigenvalues 
	$\lambda_1 = v- p'(\rho) \rho < v = \lambda_2 = \lambda_3.$

	The first field is genuinely nonlinear and the second and third field are linearly degenerate. The Riemann invariants associated with their corresponding eigenvalue are $(w,c)$, $(v,c)$ and $(v, \rho)$, respectively. We study the solution to the Riemann problem, i.e. the initial value problem with constant initial data $(\rho_0,w_0, c_0)$, for the system~\eqref{eq:ADPAR_Eulerian}. The solution, for some given constant initial data, consists of a shock or rarefaction wave associated with the first eigenvalue, followed by a contact discontinuity associated with the second and third eigenvalue (2-3-contact discontinuity). The following proposition summarizes the Riemann problem in Eulerian coordinates (see~\cite{BagRas2003} for the Lagrangian formulation):
	\begin{proposition} \label{prop:RiemannsolutionADPAR}
		Consider the Riemann problem
		
		\begin{align} \label{eq:ADPAR_RP}
		\begin{cases}
		\partial_t \colvector{\rho \\ \rho w \\ \rho c} + \partial_x \colvector{ \rho v \\ \rho w v \\ \rho c v} = 0 \\
		(\rho, w, c)(x,0) = \begin{pmatrix} (\rho^+,w^+, c^+) & \mbox{ for } x>0 \\ (\rho^-, w^-, c^-) & \mbox{ for } x \leq 0. \end{pmatrix}.
		\end{cases}
		\end{align}
		The solution $U(x,t) = (\rho, w, c)(x,t)$, to the system \eqref{eq:ADPAR_RP}, is as follows:
		\begin{itemize}
			\item [(i)] We connect $U_-$ to an intermediate state $\tilde{U} = (\tilde{\rho},\tilde{w}, \tilde{c})$ such that $\tilde{w} = w_-, \tilde{c} = c_-, \tilde{v} = v_+$ holds. The wave connecting $U_-$ and $\tilde{U}$ is either a 1-shock if $v_+ < v_-$ and a 1-rarefaction if $v_+ > v_-$. The states $\tilde{U}$ and $U_+$ are then connected by a 2-3 contact discontinuity of velocity $v_+$. 
			\item[(ii)] Additionally, $w$ and $c$ take only the values $(w_-,c_-)$ and $(w_+,c_+)$. The velocity $v$ is a monotone function of $x/t$ with $\min \lbrace v_-,v_+ \rbrace \leq v(x,t) \leq \max \lbrace v_-,v_+ \rbrace$ and for  $x> t v_+ $, we always have $U(x,t) = U_+$. 
			\item[(iii)] $U(x,t)$ and $v(x,t)$ remain in an invariant region $\mathcal{R}$ away from the vacuum
		\end{itemize}
		\begin{align*}
		\mathcal{R} := \lbrace (\rho,w,c) \lvert (v,w,c) \in [v_{\min}, v_{\max}]  \times  [w_{\min}, w_{\max}] \times  [c_{\min}, c_{\max}]\rbrace
		\end{align*}
	\end{proposition}
	Subsequently, we will refer to this  solution, with our pressure approximation, as the "Adapted Pressure ARZ model" (AP). 
	We now define the network solution. %analogously to~\Cref{def:networksolution_ARZ} 
	%with the only difference to the ARZ model that the homogenized pressure $p^*$ is substituted by the pressure law $p^{**}$.
	%
	\begin{definition}[Network solution AP] \label{def:networksolution_ADPAR}
		Consider a junction $k$ with $n$ incoming and $m$ outgoing roads, with constant initial data $U_{i,0}, i \in (\delta_{k}^- \cup \delta_{k}^+)$. We say that the family $\lbrace U_i(x,t) \rbrace_{i \in (\delta_{k}^- \cup \delta_{k}^+)}$ is an admissible solution to the Riemann problem~\eqref{eq:ARZ_RP_network} with approximated homogenized pressure if and only if
		\begin{itemize}
			\item $\forall i \in  (\delta_{k}^- \cup \delta_{k}^+)$, $U_i(x,t)$ is a weak solution of the network problem~\eqref{eq:networkproblem1} where the pressure is $ p_i ~~ \forall i \in \delta_{k}^-$. On an outgoing road $j \in \delta_{k}^+$, the solution in the triangle $\lbrace (x,t) ~ \vert ~ 0 < x < t \, v_{j,0} \rbrace$ is the approximated homogenized solution with pressure $p_j^{**}$. For $\lbrace (x,t) ~ \vert ~ t \, v_{j,0} < x \rbrace$, the pressure is $p_{j,0}$;
			\item the constraint~\eqref{eq:pressureapproximation} for the pressure is satisfied (H4*), and the homogenized $w$-value is given by $\overline{w_j} := \sum_{i \in \delta_{k}^-} \frac{q_{ji}}{q_j} w_i(U_{i,0}) ~~\forall j \in \delta_{k}^+$;
			\item the sum of the incoming fluxes is $\underline{maximal}$ subject to (H1*)-(H2) and they satisfy the mixture rule (H3).
		\end{itemize}
	\end{definition}
	 
	We remark that in comparison to the ARZ model, we only replace the true homogenized pressure $p^*$ by the approximation $p^{**}$. 
	
	Note that due to (H4*), the pressure law is always of the form $c p(\rho)$. Thus, the demand and supply function in (H1*) can be simplified to the  compact form~\eqref{eq:ADPAR_Demand}-\eqref{eq:ADPAR_Supply} %shown in~\Cref{img:ADPAR_DemandSupply}.
	\begin{align} \label{eq:ADPAR_Demand}
	D^{\ADPARZ}(\rho,w,c) = \begin{cases}
	( w - c p(\rho))\rho & \text{if } \rho \leq \sigma(w,c),\\
	( w - cp(\sigma(w,c)))\sigma(w,c) & \text{otherwise,} 
	\end{cases} \\
	S^{\ADPARZ}(\rho,w,c) = \begin{cases}
	( w - cp(\sigma(w,c)))\sigma(w,c)  & \text{if } \rho \leq \sigma(w,c),\\
	( w - cp(\rho))\rho & \text{otherwise.} 
	\end{cases} \label{eq:ADPAR_Supply}
	\end{align}

	\subsection{The case of the  n-1-junction} \label{sec:nto1junction}
	
	To exemplify, we  consider the case $i \in \delta^-_k, |\delta^-_k|=n$ and $j \in \delta^+_k, |\delta^+_k|=1.$ We describe the approximation of the pressure $p^{**}_j$. Assume that initially, the pressure is given by $p_{j,0}(\rho) = c_{j,0}  \rho^{\gamma}, j \in \delta_{k}^+$, which is the prototype pressure function used in the literature~\cite{AwRas2000,GarPic2006}. Assume that the mixture of incoming flows complies with the mixture rule $\vec\beta = (\beta_i)_{i=1, \dots, n}$ with $\sum_{i=1}^{n} \beta_i=1$. Together with the vector of incoming Lagrangian markers $\vec{w} = (w_i)_{i=1, \dots,n}$, the homogenized value $\overline{w}$~\eqref{eq:barw_j} is given by 
	\begin{equation}\label{eq:homogenizedwwithbeta}
	\overline{w_j} = \sum_{i=1}^{n} \beta_i w_i, ~~j \in \delta_{k}^+.
	\end{equation}
	We  approximate the homogenized pressure $p^*_j$   by $p^{**}_j$ as in  equation~\eqref{eq:pressureapproximation}. We define the approximated pressure therefore by 
	\begin{align*}
	p^{**}_j(\rho) = c(\vec{\beta}, \vec{w}) p_{j,0}(\rho) = c(\beta_1, \dots \beta_n, w_{1}, \dots w_{n}) c_{j,0}\rho^\gamma := \overline{c} \rho^\gamma,
	\end{align*}
	where $\overline{c}$ is given by equation \ref{cbar}.
	
	\begin{lemma} \label{lem:nto1}
		Consider a junction $k$ with $n$ incoming roads and a single outgoing road, with constant initial data $U_i = (\rho_{i,0}, w_{i,0}),~i \in (\delta_{k}^- \cup \delta_{k}^+)$, initial pressure  $p(\rho) = c_{j,0} \, \rho^\gamma, \gamma \geq 1$ on the outgoing road $j=n+1$ and the rules (H1*),(H2),(H3),(H4*). The approximation of the homogenized pressure~\eqref{eq:pressureapproximation} on the outgoing road is given by 
		\begin{equation} \label{eq:approxpressurev=0}
		p^{**}_j(\rho) = c_{j,0} \left( \sum_{i=1}^{n} \beta_i w_i \left(\sum_{l=1}^{n} \frac{\beta_l}{w_l^{1/ \gamma}}\right)^\gamma \right) \rho^\gamma.
		\end{equation}	
		Then there exists a unique network solution in the sense of Definition~\ref{def:networksolution_ADPAR}. 
	\end{lemma}
	
	\begin{proof}
		For $\gamma \geq 1$, we have $P_j(\tau) = p_j(1/\tau) = p_j (\rho) =c_{j,0} \rho^\gamma$ and $\tau$ following~\eqref{eq:adaptedpress1} is defined by the relationship
		\begin{align*}
		\tau = \int P_j^{-1}(w_{i} -v) \, \mathrm{d} \mu_x=  \sum_{i=1}^{n} \beta_i \left( \frac{c_{j,0} }{w_{i} -v} \right)^{\frac{1}{\gamma }},
		\end{align*}
		where $d\mu_x$ describes the mixture with priorities $\beta_i, i=1, \dots, n$. Due to~\eqref{eq:adaptedpress3},  the homogenized Lagrangian marker $\overline{w}$ is the sum of the velocity $v$ and the  homogenized
		pressure $p^*_j,$ i.e., $\overline{w_j} = v + p^*(\rho).$
		According to assumption (H4*) we have 
		\begin{equation} \label{eq:pressurerelation}
		\overline{w_j} = v+ p^{**}(\rho) \overset{\eqref{eq:pressureapproximation} }{=} v + c(\vec{\beta}, \vec{w}) c_{j,0} \left( \frac{1}{\tau} \right)^\gamma =: v + \overline{c} \left( \frac{1}{\tau} \right)^\gamma,
		\end{equation}
		with a function $\overline{c}$ depending only on $\vec{\beta}$ and $\vec{w}.$ Those are quantities known due to the initial data. Comparing now the value of $\overline{w_j}$ and $\overline{c}$ 
		we {\bf propose} the following choice for $\overline{c}:$
		\begin{equation}
		\overline{c} =  c_{j,0} \left(\sum_{i=1}^{n} \beta_i w_i \left(\sum_{l=1}^{n} \frac{\beta_l }{w_l^{1 / \gamma }}\right)^\gamma \right). \label{cbar}
		\end{equation}
		Once the pressure $p^{**}$ for the outgoing road with the priority rule $\vec{\beta}$ given by (H3) is determined, we can determine the flux through the junction, which is maximal subject to (H1*), with the demand and supply functions~\eqref{eq:ADPAR_Demand}-\eqref{eq:ADPAR_Supply}. Since there there is a unique outgoing road, the fluxes leaving the incoming roads have to exit the junction at the outgoing road by rule (H2). On the incoming roads, the pressure law is unchanged and the solution $U_i(x,t)$ is the solution of a (half-)Riemann problem~\eqref{eq:ARZ_RP_network}. On the outgoing road, the pressure function  $p^{**}_j(\rho)$ possesses the same properties as the prototype pressure function $p_{j,0}(\rho)$. In particular, the flux function $\rho (\overline{w_j} - p_j^{**}(\rho))$ is concave and has a unique maximum $\sigma(\overline{w_j}, \overline{c})$. The state $U_j^-$ corresponding to the flux value exiting the junction is therefore uniquely determined. Away from the junction, the solution is the unique entropy solution of the Riemann problem~\eqref{eq:ADPAR_RP} given in Proposition~\ref{prop:RiemannsolutionADPAR} with $U_j^-$ on the left and $U_j^+ = U_{j,0}$ on the right. 
	\end{proof}
	
	The crucial point in the previous proof is the definition of $\overline{c}$ given by equation \eqref{cbar}.  This choice is motivated by the following
	consideration. According to the result in \cite{HerRas2006} the homogenized pressure $p^*_j$ for a fixed value $\overline{w_j}$ the relation 
	between $v,w$ and $\rho$ on road $j$ can be expressed as 
	\begin{equation} \label{eq:pressurerelation1}
	w = v+ p^{*}_j(\rho). 
	\end{equation}
	We may use the general  Ansatz $p^*(\rho)= C(\rho,v) \rho^\gamma$ to define $p^*_j.$ Hence, the choice \eqref{cbar} implicitly assumes that $C$ is {\em independent } of $\rho.$ 
	Then  we may proceed as in the previous and obtain the following functional dependence $C=C(v)$  
	\begin{align}\label{C22}
	\overline{C}(v) &= (\overline{w_j} -v) \left( \sum_{i=1}^{n} \beta_i \left( \frac{c_{j,0} }{w_{i} -v} \right)^{\frac{1}{\gamma }} \right)^\gamma . 
	\end{align}
	Comparing now \eqref{C22} and \eqref{cbar} we observe that the proposed choice is simply  $\overline{c}=C(0),$ i.e., setting $v=0$ in equation \eqref{C22}.  
	Obviously, other choices are possible. However, for the presented approach it is crucial that the final approximation  $p^{**}_j$ contains a value $\overline{c}$
	independent of the dynamic quantities$(\rho,v,w).$ Only  this fact allows to propagate the value $\overline{c}$ with velocity $v$ and only this fact
	justifies the additional equation in Proposition~\ref{prop:RiemannsolutionADPAR}.

	%%%%%%%%%%%%%%%%%%%%%%%%%%%%%%%%%%%%%%%%%%%%%%%%%%%%%%%%%%%%%%%%%%%%%%%%%%%%%%%%%%%%%%%%%%%%%%%%%%%%%%%%%%%%%%%%%%%%%%%%%%%%%%%%%%%%%%%%%%%%%%%%%%%%%%%%%%%%%%%%%%%%%%%%%%%%%%%%%%%%%%%%%%%%%%%%
	
	\textit{Example: The 2-1-junction, $\gamma =1$:}
	We state the explicit formula for $p^{**}$ for a junction  merging of two incoming roads $i=1,2$ into a single outgoing road $j=3$
	%, as shown in~\Cref{img:Network21}, 
	Assume that fixed ratios $\beta_1 = \beta$ and $ \beta_2 = 1-\beta$ are assigned to the incoming roads such that the homogenized the Lagrangian marker is given by~\eqref{eq:homogenizedwwithbeta} and the pressure according to Lemma 1 by~\eqref{eq:approxpressure_21}.
	\begin{align}
	\overline{w_3} &= \beta_1 w_{1,0} + \beta_2 w_{2,0},  \label{eq:ADPAR_w3} \\
	p^{**}_3(\rho) &= c_{3,0} \left( \sum_{i=1}^{2} \sum_{l=1}^{2}\frac{\beta_i \beta_l w_{l,0}}{w_{i,0}} \right) \rho = c_{3,0} \left(1 + \beta (1- \beta) \frac{ \left( w_{1,0} - w_{2,0} \right)^2}{w_{1,0}w_{2,0}} \right) \rho. \label{eq:approxpressure_21}
	\end{align}
	
	For the choice $\beta=0$ or $\beta=1$  the   homogenized pressure $p^*_3$ can be computed exactly and we compare it with the  approximation $p^{**}_3$.  In this case $p^*_3\equiv p^{**}_3.$  Furthermore, for  $w_{1,0} = w_{2,0}$ and any value of $\beta$ we have  $p^{**}_3(\rho) = p^{*}_3(\rho) = p_{3,0}(\rho)$.  Figure~\ref{img:pressurecomparison_all} illustrates {\bf numerically }the differences in pressure approximations for $\beta =0.5, w_{1,0} = 9/2, w_{2,0} = 7/2$ and $\overline{w_3}= 4$. Initially, the pressure is set to $p_{3,0}(\rho) = \rho$. The  homogenized  pressure, $p^*_3(\rho)$ is numerically computed by the homogenization formula~\eqref{eq:adaptedpress2} and the pressure $p^{**}_3(\rho)$ is obtained from~\eqref{eq:approxpressure_21}. 
	
	\begin{figure}[tbhp]
		
		\centering
		\begin{subfigure}{0.45\textwidth}
			\input{./img/derivationofadaptedpressure/zzz_pressure_difference}
			\subcaption{Approximation error}
			\label{img:pressurecomparison_pressuredifference}
		\end{subfigure}
		\begin{subfigure}{0.45\textwidth}
			\input{./img/derivationofadaptedpressure/zzz_fluxwithdifferentpressures}
			\subcaption{Flux functions with different pressures}
			\label{img:pressurecomparison_flux}
		\end{subfigure}
		
		\caption{Homogenized pressure $p^*$ and its approximation $p^{**}$ as well as the  initial pressure $p_0$.}
		\label{img:pressurecomparison_all}
	\end{figure}
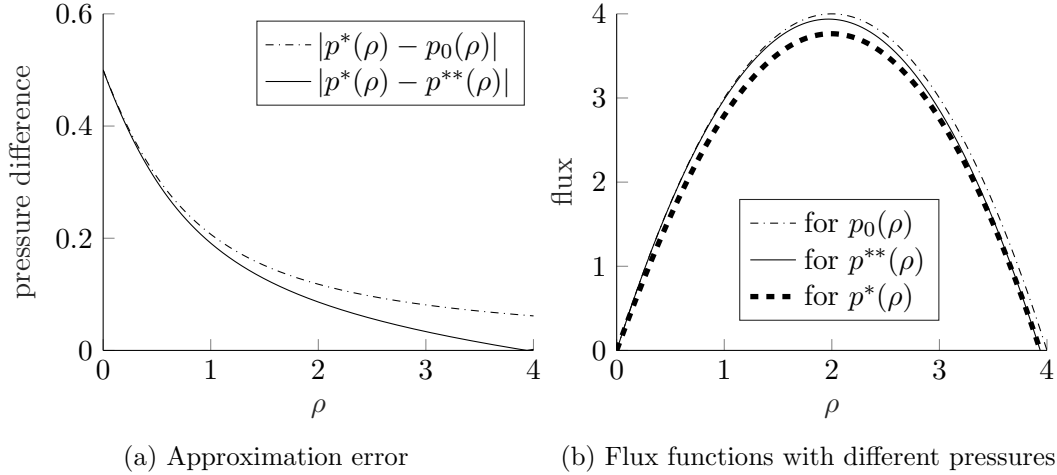

	\subsection{The case of a n-m-junction}
	
	Consider a junction with $i=1, \dots n$ incoming roads and $j=1, \dots m $ outgoing roads. Using distribution rates $\alpha_{ji}$ and priority coefficients $\beta_{ij}$ complying with the rules (H2)-(H3), we determine the solution on each outgoing road using a construction similar to the $1$-$n$ junction.
	
	\begin{lemma} \label{lem:pressure_ntom}
		Consider a junction with $n$ incoming roads and $m$ outgoing roads, with constant initial data $U_i = (\rho_{i,0}, w_{i,0}),~i \in (\delta_{k}^- \cup \delta_{k}^+)$, initial pressures  $p_j(\rho) = c_{j,0} \, \rho^\gamma, \gamma \geq 1$ on the outgoing roads  $j=n+1, \dots n+m$ and the rules (H1*), (H2), (H3), (H4*). Then, there exists a unique network solution according to Definition~\ref{def:networksolution_ADPAR} and the approximation of the homogenized pressure \eqref{eq:pressureapproximation} on each outgoing road, setting  $v=0$, is given by
		\begin{equation} \label{eq:pressure_ntom}
		p^{**}_j(\rho) = c_{j,0} \left( \sum_{i=1}^{n} \beta_{ij} w_i \left(\sum_{l=1}^{n} \frac{\beta_{lj}}{w_l^{1 / \gamma}}\right)^\gamma \right) \rho^\gamma.
		\end{equation}
	\end{lemma}
	
	\begin{proof}
		Once the ratios of the incoming fluxes $(\beta_{ij})_{i \in \delta_{v}^-}$ for an outgoing road $j$ are fixed, the approximation of the homogenized pressure is determined using Lemma~\ref{lem:nto1}. On each outgoing road $j$, the pressure is then given by~\eqref{eq:pressure_ntom}.
	\end{proof}
	%%%%%%%%%%%%%%%%%%%%%%%%%%%%%%%%%%%%%%%%%%%%%%%%%%%%%%%%%%%%%%%%%%%%%%%%%%%%%%%%%%%%%%%%%%%%%%%%%%%%%%%%%%%%%%%%%%%%%%%%%%%%%%%%%%%%%%%%%%%%%%%%%%%%%%%%%%%%%%%%%%%%%%%%%%%%%%%%%%%%%%%%%%%%%%%%%%%%%%%%%%%%%%%%%%%%%%%%%%%%%%%%%%%%%%%%%%%%%%%%%%%%%%%%%%%%%%%%%%%%%%%%%%%%%%%%%%%%%%%%%%%%%%%%%%%%%%%%%%%%%%%%%%%%%%%%%%%%%%%%%%%%%%%%%%%%%%%%%%%%%%%%%%%%%%%%%%%%%%%%%%%%%%%%%%%%%%%%%%%%%%%%%%%%%%%%%%%%%%%%%%%%%%%%%%%%%%%%%%%%%%%%%%%%%%%%%%%%%%%%%%%%%%%%%%%%%%%%%%%%%%%%%%%%%%%%%%%%%%%%%%%%%%%%%%%%%%%%%%%%%%%%%%%%%%%%%%%%%%%%%%%%%%%%%%%%%%%%%%%%%%%%%%%%%%%%%%%%%%%%%%%%%%%%%%%%%%%%%%%%%%%%%%%%%%%%%%%%%%%%%%%%%%%%%%%%%%%%%%%%%%%%%%%%%%%%%%%%%%%%%%%%%%%%%%%%%%%%%%%%%%%%%%%%%%%%%%%%%%%%%%%%%%%%%%%%%%%%%%%%%%%%%%%%%%%%%%%%%%%%%%%%%%%%%%%%
	\section{Numerical solution procedure for the traffic network} \label{sec:numscheme}
	
	We introduce the following notations used for the numerical solution procedure. Let $Y_i(x,t) = (\rho_i,\rho_i w_i,\rho_i c_i)(x,t)$ denote the traffic state in conservative variables on road $i$ at position $x$ and time $t$, whereas $f(Y_i) = (\rho_i v_i, \rho_i w_i v_i, \rho_i c_i v_i)$ denotes the flux. We introduce a grid in time and space with step sizes $\Delta t$ and $\Delta x$ to discretize the system~\eqref{eq:ADPAR_Eulerian}. Each road is divided into $N_{x_i} = L_i/ \Delta x$ cells of equal size and we consider a finite number of time discretizations $N_t = T/ \Delta t$. Let $I_{i,j}$ denote the open interval $(x_{i,j-0.5}, x_{i,j+0.5})$ for $j=1, \dots N_{x_i}$, and let $Y_{i,j}^s$ denote the average value of the function $Y_i(x,t)$ on the interval $I_{i,j}$ at time $t^s = s \Delta t$, i.e.,
	\begin{align*}
	Y_{i,j}^s =  \frac{1}{\Delta x } \int_{I_{i,j}} Y_{i,j}(x,t^s) \, \mathrm{dx},
	\end{align*}
	such that the following CFL condition is satisfied:
	\begin{align}
	\frac{\Delta t}{\Delta x} \max_i \max_j \lbrace \abs{\lambda_l(Y_{i,j}^s)}, l=1,2,3 \rbrace\leq \frac{1}{2} \qquad s=1, \dots, N_t. 	\label{eq:CFL}
	\end{align}
	
	\subsection{The Transport-Equilibrium scheme}
	
	We discretize~\eqref{eq:ADPAR_Eulerian} using a trans\-port-equilibrium scheme based on a Godunov discretization. The analysis in~\ref{sec:pressurederivation} highlighted the importance of the description of the Lagrangian marker since its variation results in the recalculation of the pressure. The Godunov scheme is known to have unphysical oscillations, at least for realistic grid sizes, which reduce the quality of the numerical solution. In~\cite{ChaGoa2007}, it was shown, that the scheme does not comply with the maximum principle on the Riemann invariants for the system~\eqref{eq:ARZ}. To better treat the 2-3-contact discontinuity and ensure the correct depiction of the Riemann invariant $w$, we use the transport equilibrium (TE) scheme~\cite{ChaGoa2007} and expand it to the AP system. The TE scheme is a wave splitting strategy where the contact discontinuities and the 1-waves evolve separately. The first step only accounts for the contact discontinuity and the second step focuses on the 1-wave. Let $\tilde{Y}(Y_-,Y_+)$ denote  the intermediate state in conservative variables in the solution to the Riemann problem in Section~\ref{sec:homogenizedsystem} with initial states $Y_-$ on the left and $Y_+$ on the right. The first step is based on Glimm's random sampling strategy. The intermediate value $Y_{i,j}^{s+1/2}$ is determined by means of a well distributed random sequence $({\alpha}_s)$ within $(0,1)$ and is set to
	\begin{align*}
	Y_{i,j}^{s+ \nice12} = \begin{cases}
	\tilde{Y}(Y_{i,j-1}^{s}, Y_{i,j}^{s}) &\text{ if } {\alpha}_{s+1} \in (0,  \frac{\Delta t}{\Delta x}v_{i,j}^s), \\
	Y_{i,j}^{s} &\text{ if } {\alpha}_{s+1} \in [\frac{\Delta t}{\Delta x} v_{i,j}^s ,1).
	\end{cases}
	\end{align*}
	Here, we use the van der Corput random sequence defined by 
	$
	{\alpha}_s = \sum_{l=0}^{z} i_l 2^{-(l+1)},
	$
	computed using the binary expansion of the integer $s = \sum_{l=0}^{z} i_l 2^l$, $i_l \in \lbrace 0,1 \rbrace$. In other words, the random sampling decides whether the intermediate value is set to the initial state or to the state of a possibly present contact discontinuity. In either case, the correct fluxes through the cell interfaces have to be determined in the second step. In the absence of a contact discontinuity, the scheme is equivalent to the Godunov scheme. Assuming the CFL condition~\eqref{eq:CFL}, the complete scheme is given by 
	\begin{align}
	Y_{i,j}^{s+1} &= Y_{i,j}^{s} - \frac{\Delta t}{\Delta x} \left( F_{i,j+\nice12}^{s+\nice12,L} - F_{i,j-\nice12}^{s+\nice12,R} \right), \label{eq:TEscheme}
	\end{align}
	where the left and right numerical flux functions are defined as follows
	\begin{align*}
	F_{i,j+\nice12}^{s+\nice12,L} &= G_{i,j+\nice12}^s(Y_{i,j}^{s+\nice12}, Y_{i,j+1}^s), \\
	F_{i,j-\nice12}^{s+\nice12,R} &= \begin{cases}
	G_{i,j-\nice12}^s(Y_{i,j-1}^s, Y_{i,j}^{s+\nice12}) &\text{ if } \tilde{Y}(Y_{i,j-1}^{s}, Y_{i,j}^{s+\nice12}) = Y_{i,j}^{s+\nice12}, \\
	f(Y_{i,j}^{s+\nice12}) &\text{ otherwise}.
	\end{cases}
	\end{align*}
	
	The flux terms $G$ are the usual Godunov fluxes given by
	\begin{align*}
	G_{i,j+\nice12}^s &= G_{i,j+\nice12}^s(Y_{i,j}^s, Y_{i,j+1}^s) =\colvector{q_{i,j+\nice12}^s \\ w_{i,j}^s q_{i,j+\nice12}^s \\ c_{i,j}^s q_{i,j + \nice12}^s }, \notag \\ 
	\text{ where } q_{i,j+\nice12}^s &= \min\lbrace D^{\ADPARZ}(\rho_{i,j}^s, w_{i,j}^s, c_{i,j}^s), S^{\ADPARZ}(\tilde{\rho}_{i,j+1}^s, w_{i,j}^s, c_{i,j}^s) \rbrace, %\label{eq:Godunov_ADPAR}
	\end{align*}
	and $\tilde{\rho}_{i,j+1}^s$ is either given by the intersection of the curves $\lbrace w_i(Y) = w_{i,j}^s, c_i(Y) = c_{i,j}^s \rbrace$ and $\lbrace v_i(Y) = v_{i,j+1}^s \rbrace$, if this intersection exists; otherwise $\tilde{\rho}_{i,j+1}^s$ is set to zero. 
	
	\subsubsection{The boundary conditions}
	
	The flux terms $F^{s+\nice12,R}_{i,j-\nice12}$ for $j=1$ and $F^{s+\nice12,L}_{i,j+\nice12}$ for $j=N_{x_e}$, in the TE scheme~\eqref{eq:TEscheme}, are obtained by coupling the boundary conditions. The flux at the junction is determined by the coupling condition, which depends upon the type of the junction. We illustrate the computation of the boundary fluxes at the 2-1-junction %shown in~\Cref{img:Network21} 
	and the Riemann data 
	%\begin{equation*}
	$Y_{i,0} = (\rho_{i,0}, \rho_{i,0} w_{i,0}, \rho_{i,0} c_{i,0}), i=1,2,3$.
	%\end{equation*} 
	Assume that the priorities are $\beta$ and $1-\beta$ and that the pressure on the outgoing road is $p_{3,0}(\rho) = c_{3,0} \rho$. To determine the solution to the Riemann problem according to Definition~\ref{def:networksolution_ADPAR}, we proceed as follows: We compute the incoming $w$-value $\overline{w_3}$ given by~\eqref{eq:ADPAR_w3} and the new constant for the pressure function $\overline{c_3} = c_{3,0}  c( (\beta, 1- \beta), (w_{1,0}, w_{2,0})) $ defined by Lemma~\ref{lem:nto1} once at $t=0$. The demand and the supply at the junction are
	\begin{align}\label{eq:ADPAR_demandsupply_21}
	D_1 = D^{\ADPARZ} (\rho_{1,0},w_{1,0}, c_{1,0}) \quad
	D_2 = D^{\ADPARZ} (\rho_{2,0},w_{2,0}, c_{2,0}) \quad
	S_3 = S^{\ADPARZ}(\tilde{\rho}, \overline{w_3},\overline{c_3}),
	\end{align}
	where $\tilde{\rho}$ is either given by the intersection of the curves $\lbrace w(Y) = v + p^{**}(\rho) = \overline{w_3}, c_3(Y) = \overline{c_3} \rbrace$ and $\lbrace v(Y) = v_{3,0} \rbrace$ or $\tilde{\rho} = 0$. The flow into the outgoing road is given by
	\begin{align*} %\label{eq:ADPAR_fluxes21junction}
	\colvector{q_3 \\ \overline{w_3} q_3 \\ \overline{c_3} q_3 } ~~~ \text{where } q_3 = \min \lbrace D_1 / \beta, D_2 / (1- \beta), S_3 \rbrace, 
	\end{align*}
	and the flows at the end of the two incoming roads are given by
	\begin{align*} %\label{eq:ADPAR_fluxes21junction}
	\colvector{q_1 \\ w_1 q_1 \\ c_1 q_1 } ~~~ \text{where } q_1 = \beta q_3 \qquad \qquad
	\colvector{q_2 \\ w_2 q_2 \\ c_2 q_2 } ~~~ \text{where } q_2 = (1-\beta) q_3. 
	\end{align*}
	
	The new pressure law on the outgoing road is $p^{**}(\rho) =\overline{c_3} \rho$. We illustrate now, how to determine the boundary states for $s=0$. For the computation with the Godunov scheme, the fluxes $q_1,q_2$ and $q_3$ would be sufficient, since the scheme is based purely on flux terms. In contrast, the computation of the flux terms  $F^{\nice12,R}_{i,\nice12}$ and $F^{\nice12,L}_{i,N_{x_i}+\nice12}$, in the TE scheme, requires the density values of the states $Y_{3,0}^0$ and $Y_{1,N_{x_1}}^0,Y_{2,N_{x_2}}^0 $. We begin with the computation of $Y_{3,0}^0$. Once the flux $q_3$ is determined, we can compute two states for which 
	\begin{equation}\label{eq:TEscheme_boundaryvalue1}
	f_1(Y) = q_3. 
	\end{equation}
	Since waves on the outgoing road must have positive velocities, the state $Y_{i,0}^s$ is uniquely determined as the state such that
	\begin{equation} \label{eq:TEscheme_boundaryvalue2}
	\lambda_1(Y) \geq 0.
	\end{equation}
	We can obtain the boundary state $Y_{3,0}^0 =(\rho, \rho w, \rho c)$ defined by~\eqref{eq:TEscheme_boundaryvalue1}-\eqref{eq:TEscheme_boundaryvalue2} as follows.
	\begin{align*}
	\rho = \frac{\overline{w_3}}{2 \overline{c_3}} - \sqrt{\left( \frac{\overline{w_3}}{2 \overline{c_3}} \right)^2 - \frac{q_3}{\overline{c_3}}}, \qquad w = \overline{w_3}, \qquad c = \overline{c_3}.
	\end{align*}
	The flux term $F_{3,\nice12}^{\nice12,R}$ can then be computed using random sampling and the states $Y_{3,0},Y_{3,1}$, see~\eqref{eq:TEscheme}. Analogously, we proceed for the end of road $i=1,2$, where we have to choose the state $Y_{i,N_{x_i}+1}$ with negative first eigenvalue
	\begin{equation} \label{eq:TEscheme_boundaryvalue3}
	\lambda_1(Y) \leq 0.
	\end{equation}
	Thus, for road 1, the boundary state $Y_{1,N_{x_1}}^0$ is given by 
	\begin{align*}
	\rho = \frac{w_1}{2 c_1} + \sqrt{\left( \frac{w_1}{2 c_1} \right)^2 - \frac{q_1}{c_1}}, \qquad w = w_1, \qquad c = c_1,
	\end{align*}
	and one can proceed similarly for the second road. The procedure for obtaining the boundary states for other junctions is identical. Once the flux $q_i$ for road $i$ at a junction is determined, we can compute the two states $Y$, such that $f_1(Y) = q_i$ and choose the correct state $Y$ depending on whether the boundary state is evaluated at the beginning or the end of a road $i$. 
	\begin{algorithm}
		\caption{Numerical simulation of a network}\label{alg:networksimulation}
		\begin{algorithmic}[1]
			\REQUIRE Roads $i=1,..,N_e$, nodes $k=1,..,N_v$. Initial data $\rho_i(x,0),w_i(x,0),c_i(x,0)$
			\ENSURE Densities $\rho_i(x,t)$, Lagrangian marker $w_i(x,t)$ and pressure coefficients $c_i(x,t)$ for all times $t^s = s \Delta t$ with $s \in \{1,\dots, N_t \}$ 
			\FOR {$s = 0,1, \dots, N_t$}
			\FOR {$k = 1, \dots, N_v$} 
			\IF { $k$ is a merging junction and $s = 0$ or  $\overline{w}^{s-1} \neq \overline{w}^s$}
			\STATE Update coefficient for the pressure according to Lemma~\ref{lem:nto1}.
			\ELSE 
			\STATE Use pressure coefficient of $t^{s-1}$.
			\ENDIF 
			\STATE Compute the maximum flux at the junction at time $t^s$ with demand and supply~\eqref{eq:ADPAR_Demand}-\eqref{eq:ADPAR_Supply} respecting the coupling condition~\eqref{eq:ADPAR_demandsupply_21}.
			\ENDFOR
			\FOR {$i = 1, \dots N_e$}
			\STATE Determine $a_{s+1}$ and the boundary states $Y^s_{i,0}, Y^s_{i,N_{x_i}+1}$ according to~\eqref{eq:TEscheme_boundaryvalue1}-\eqref{eq:TEscheme_boundaryvalue3} from the fluxes at the junction. Compute the solution at time $t^{s+1}$ with the TE scheme~\eqref{eq:TEscheme}.
			\ENDFOR
			\ENDFOR
		\end{algorithmic}
	\end{algorithm}

	%%%%%%%%%%%%%%%%%%%%%%%%%%%%%%%%%%%%%%%%%%%%%%%%%%%%%%%%%%%%%%%%%%%%%%%%%%%%%%%%%%%%%%%%%%%%%%%%%%%%%%%%%%%%%%%%%%%%%%%%%%%%%%%%%%%%%%%%%%%%%%%%%%%%%%%%%%%%%%%%%%%%%%%%%%%%%%%%%%%%%%%%%%%%%%%%%%%%%%%%%%%%%%%%%%%%%%%%%%%%%%%%%%%%%%%%%%%%%%%%%%%%%%%%%%%%%%%%%%%%%%%%%%%%%%%%%%%%%%%%%%%%%%%%%%%%%%%%%%%%%%%%%%%%%%%%%%%%%%%%%%%%%%%%%%%%%%%%%%%%%%%%%%%%%%%%%%%%%%%%%%%%%%%%%%%%%%%%%%%%%%%%%%%%%%%%%%%%%%%%%%%%%%%%%%%%%%%%%%%%%%%%%%%%%%%%%%%%%%%%%%%%%%%%%%%%%%%%%%%%%%%%%%%%%%%%%%%%%%%%%%%%%%%%%%%%%%%%%%%%%%%%%%%%%%%%%%%%%%%%%%%%%%%%%%%%%%%%%%%%%%%%%%%%%%%%%%%%%%%%%%%%%%%%%%%%%%%%%%%%%%%%%%%%%%%%%%%%%%%%%%%%%%%%%%%%%%%%%%%%%%%%%%%%%%%%%%%%%%%%%%%%%%%%%%%%%%%%%%%%%%%%%%%%%%%%%%%%%%%%%%%%%%%%%%%%%%%%%%%%%%%%%%%%%%%%%%%%%%%%%%%%%%%%%%%%
%	
	\section{Computational results} \label{sec:computationalresults}
	
	In this section, we provide  numerical results  to compare our solution with  the homogenized solution  using the TE scheme. We also provide  numerical results  in the case of time-dependent
	boundary data and a comparison to the  LWR solution. 
	
	\subsection{The Riemann problem at a 2-1 junction}

	We investigate the 2-1-junction and compare our network solution to the solution with true homogenized pressure to show that we can indeed approximate the truly homogenized solution.Figure~\ref{img:Riemann21_HerMouRas_comparison} depicts the solution to a Riemann problem with $\vec{\beta} = (0.5, 0.5)$. The initial data is $(\rho_1, w_1, c_1) = (3,2,1), (\rho_2,w_2, c_2) = (2,1, 1), (\rho_3, w_3, c_3) = (3,2,1)$. As shown in Figure~\ref{img:Riemann21_HerMouRas_comparison},  the solution on the outgoing road with the approximated pressure $p^{**}$ is close to the solution using  pressure $p^*$, determined by the homgenization process described in~\cite{HerMouRas2006}. Moreover, the exact solution of the Riemann problem with the approximated pressure is compared with the solution given by the TE scheme in Figure~\ref{img:Riemann21_gridconvergence}. Numerically, the pressure coefficient is changed once at $t=0$ according to \eqref{eq:approxpressure_21} for the Riemann problem since after the initial interaction all waves emerge from the junction. The fluxes at the junction are computed using~\eqref{eq:ADPAR_demandsupply_21}. The simulation procedure to obtain the solution on the whole network is summarized in Algorithm~\ref{alg:networksimulation}. For decreasing $\Delta x$ and $ \Delta t = \Delta x / 10$ fulfilling the CFL condition~\eqref{eq:CFL}, we see convergence towards the exact solution of the Riemann problem as Figure~\ref{img:Riemann21_gridconvergence} illustrates for the 1-wave at the junction and the contact discontinuity in the numerical solution at $t=0.12$.

	\begin{figure}[tbhp]
		\centering
		\begin{subfigure}{0.45\textwidth}
			\input{./img/NumericalExamples/zzz_Riemann21_HerMouRas_comparison}
			\subcaption{True homogenized and AP solution}
			\label{img:Riemann21_HerMouRas_comparison}
		\end{subfigure}
		\begin{subfigure}{0.45\textwidth}
			\input{./img/NumericalExamples/zzz_Riemann21_convergence}
			\subcaption{Grid solutions and AP solution}
			\label{img:Riemann21_gridconvergence}
		\end{subfigure}
		\caption{Solution on the outgoing road for $\beta = (0.5, 0.5)$.}
	\end{figure}
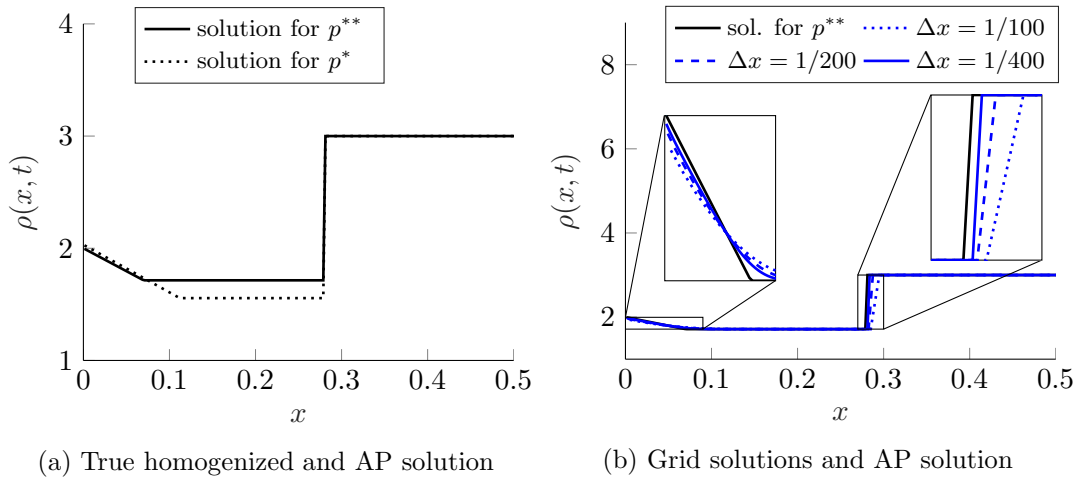
	
	\subsection{Network of merging junctions}
	
	We consider a network, which consists of a sequence of 2-1-junctions (in total $N$ merging junctions), see Figure~\ref{img:NetworkSequential}. Moreover, we consider Riemann data $(\rho_i,w_i,c_i)$ on each road $i$. The homogenized Lagrangian marker on each road $i$ is denoted by $\overline{w_i}$. The priorities at merges are set to $\beta$ and $(1-\beta)$ with $0 < \beta < 1$. We do not consider $\beta \in \lbrace 0,1 \rbrace$ here since the network would reduce to a network with consecutive 1-1-junctions for which there is no need to adapt the pressures (compared to~\ref{sec:nto1junction}).

	\begin{figure}[tbhp]
		\centering
		\resizebox{0.8\columnwidth}{!}{
			{\normalsize
				\begin{tikzpicture}[every node/.style={text width = 0.65cm}] 
				\node (B) at (3, -1) [circle,draw , align=center] {B};
				\node (C) at (6, -1) [circle,draw, align=center] {C};
				\node (D) at (9, -1) [circle,draw, align=center] {D};
				
				\node (E) at (0, -3) [circle,draw, align=center] {A };
				\node (F) at (3, -3) [circle,draw, align=center] {M1};
				\node (G) at (6, -3) [circle,draw, align=center] {M2};
				\node (H) at (9, -3) [circle,draw, align=center] {MN};
				\node (I) at (12, -3) [circle,draw, align=center] {E};
				
				\draw[->, very thick] (E) to node [below,text width = 2cm, align=center] { \small Road 0, $\beta$ } (F);
				\draw[->, very thick] (F) to node [below,text width = 2cm, align=center] {\small Road 1, $\beta$ }(G);
				\draw[->, very thick] (G) to node [below,text width = 2cm, align=center] {\small ... }(H);
				\draw[->, very thick] (H) to node [below,text width = 2cm, align=center] {\small Road N }(I);

				\draw[->, very thick] (B) to node [left,text width = 2cm, align=center] {  \small Road $N+1$, \\ $(1-\beta)$} (F);
				\draw[->, very thick] (C) to node [left,text width = 2cm, align=center] {\small Road $N+2$, \\ $(1-\beta)$ }(G);
				\draw[->, very thick] (D) to node [left,text width = 2cm, align=center] { \small Road $2N$, \\$(1-\beta)$} (H);
				
				\end{tikzpicture}
		}}
		\caption{Network of merging junctions.}
		\label{img:NetworkSequential}
	\end{figure}
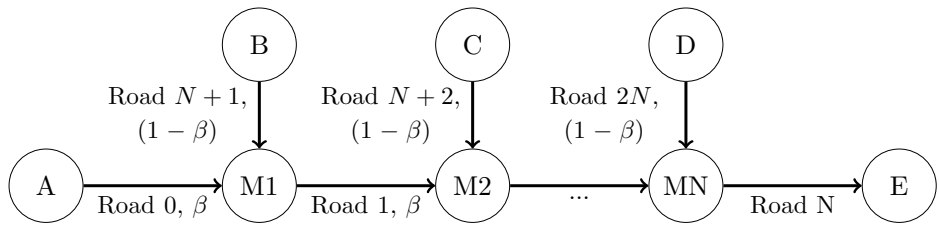
	
	The network in Figure~\ref{img:NetworkSequential} is used to investigate how a pertubation in the Lagrangian marker propagates through the network. Initially, the Lagrangian marker is identical on all roads, except for one single incoming road. We study the wave structure and the propagation speed of the perturbation. 
	
	\begin{lemma} \label{lem:Sequential}
		Consider the network shown in Figure~\ref{img:NetworkSequential} with priorities $\beta$ and $(1-\beta)$, initial data $U_{i,0} =(\rho_{i,0}, w_{i,0}, c_{i,0}) $ with $\rho_i, w_i, c_i >0,~i=0,\dots 2N$. We set initial conditions $w_{0,0} = b$ and $w_{i,0} = a, i =1,\dots, 2N,~a,b>0$ and pressures $p_i(\rho) = c_{i,0} \rho $. At a merging junction $l \in \lbrace 1, \dots N \rbrace$,  the time $t_l^*$, where $p^{**}_l \not = p_{l,0}$  depends on the initial conditions and may be infinite. The approximated  pressure on road $l$ is given by  
		
		\begin{align} \label{eq:lemSequential_pressure}
		p_l^{**}(\rho) =  c_{l,0} \left(1 + \beta (1- \beta) \frac{ \left( a - \overline{w}_{l-1} \right)^2}{a \overline{w}_{l-1}} \right)  \rho =: c_{l,0} \, d_l \,\rho,
		\end{align}
		
		where $(d_l)_l$ is monotonically decreasing sequence and $\lim_{l \rightarrow \infty} d_l = 1$.
	\end{lemma}
	
	\begin{proof}
		The propagation of the contact discontinuity through the network determines the interaction time $t_l^*$ of the contact discontinuity with the junction $l$. Without loss of generality, let us consider the dynamics on road 1 after the merging junction. At the merging junctions M1 and M2, there are two interactions to be considered (see Figure~\ref{img:Sequential_Cases}). If the contact discontinuity arising from M1 with positive speed $v_1$ and a shock with negative speed arising from the junction M2 interact, the speed of the contact discontinuity decreases after the interaction and is still non-negative, but possibly zero. The new speed is dependent on the speed of the state on the right of the shock front. In addition to a 1-shock, a 1-rarefaction can arise from the merging junction M2 with positive speed. The contact discontinuity arising from the junction M1 propagates to the junction M2. It changes its speed in the presence of a 1-shock on road 1. In the absence of a 1-shock, its velocity stays unchanged (see Figure~\ref{img:carpath_1rarefaction}). 
		\begin{figure}[tbhp]
			\centering
			\begin{subfigure}{0.45\textwidth}
				\input{./img/NumericalExamples/Sequential/carpath_1shock.tex}
				\subcaption{Interaction with a shock wave}
				\label{img:carpath_1shock}
			\end{subfigure}
			\begin{subfigure}{0.45\textwidth}
				\input{./img/NumericalExamples/Sequential/carpath_1rarefaction.tex}
				\subcaption{Interaction with rarefaction wave}
				\label{img:carpath_1rarefaction}
			\end{subfigure}
			
			\caption{Interaction of a contact discontinuity with 1-waves.}
			\label{img:Sequential_Cases}
		\end{figure}
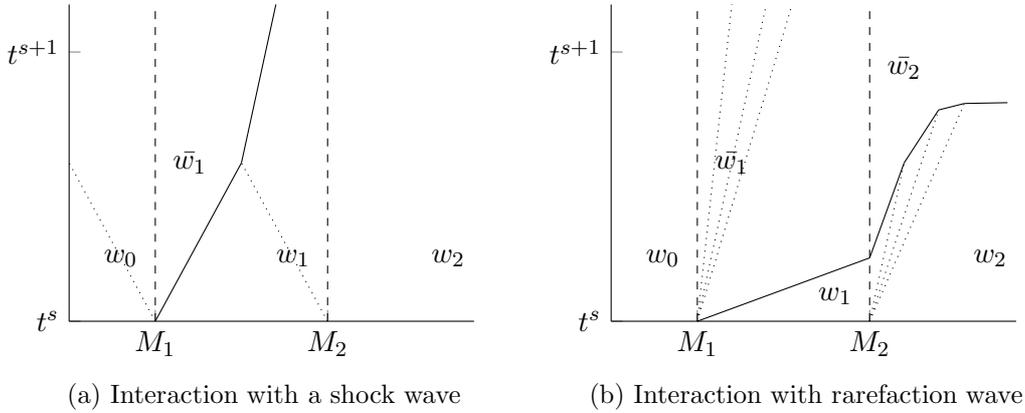	
		
		The contact discontinuity changes its speed at the junction M2 and afterwards interacts with the 1-rarefaction. During the interaction with the rarefaction fan, the contact discontinuity accelerates. Therefore, the times $t_l^*,~l=1, \dots N$, when the changes in the pressure are triggered at the junctions, depend on the initial conditions. The homogenized $w$-value on road $l$, after an interaction is triggered, is given by 
		\begin{align*}
		\overline{w}_l = \beta^{l}b + (1-\beta)a \sum_{i=0}^{l-1} \beta^{l-i-1}.
		\end{align*}
		The sequence $(\overline{w}_l)$ is either monotonically increasing ($a>b$) or decreasing ($b>a$) and $\lim_{l \rightarrow \infty} w_l = a$. At each junction, when the contact discontinuity reaches the junction, the new pressure is computed using Lemma~\ref{lem:nto1}. Then, the new pressure law is given by~\eqref{eq:lemSequential_pressure} where $\overline{w}_0 = b$. 
		Moreover, $(d_l)$ is monotonically decreasing and $\lim_{l \rightarrow \infty} d_l = 1$. Summarizing, $d_1 > d_2 > \dots > d_N$ and in the limit $N \rightarrow \infty$, we have that the pressure on the outgoing road after the last merging junction remains unchanged.
	\end{proof}
	
	Tracking the contact discontinuity is analogous to car path tracking in road networks, see also~\cite{BrePic2008}, with the only difference that each interaction of the contact discontinuity with a 1-wave triggers an additional 1-wave in our case, which has been omitted in the figures for sake of simplicity.Figure ~\ref{img:Disturbance_OneRoad_xt} exemplifies on an $x$-$t$-plane  the solution may look like. The contact discontinuity starts at the junction M1 and propagates at positive speed. 
	Whenever, the contact discontinuity interacts with a 1-wave, the speed of the contact discontinuity changes. Therefore, the new speed depends on the initial data on the roads that are located behind the merges. After an interaction with a rarefaction wave, the speed increases. After an interaction with a shock, the speed decreases. It is also possible that junctions in the network are not reached by the contact discontinuity, if the speed of the contact discontinuity approaches zero after an interaction with a shock wave. 
	In Figure~\ref{img:Disturbance_OneRoad_xt}, the contact discontinuity interacts with a shock wave on road 2, and its speed approaches zero. Therefore, the Lagrangian marker beyond the junction M3 stays unchanged. 
	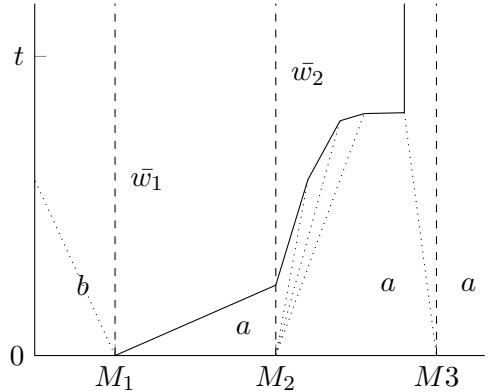
\begin{figure}[tbhp]
		\centering
		\input{./img/NumericalExamples/Sequential/disturbance_xtdiagram.tex}
		\caption{Perturbation of the Lagrangian marker at one incoming road.}
		\label{img:Disturbance_OneRoad_xt}
	\end{figure}

	\subsubsection{Numerical evaluation}
	
	We set $a = 2, b=1, \beta=\frac12, N = 10$ and $c_{i,0}=1$.  Each road has length $L_i = L = 0.5$ and the end of the time horizon is $T=12$. We consider two scenarios. In the first scenario, we set initial densities on all roads to $\rho_{i,0} = 0.3, \; i =1, \dots, N$ (free flow case). Note that we will refer to this case as the 'free flow case' to indicate the corresponding initial data.  
	\par 
	The second scenario considered is identical to the first one, except for the density on the last road, where we consider maximal congestion $\rho_{N,0} = 1$ (congested case). Since the last road is congested, a shock wave moves backwards through the network. Simulation results are obtained with the discretization $\Delta x = 1/100$ and $\Delta t = \Delta x/4$, fulfilling the CFL condition~\eqref{eq:CFL}. Table~\ref{tab:Sequential_adaptiontimes} shows the adaption times $t_l^*$ at the junctions $l = 1, \dots, 10$. We see that the pressure is recalculated\footnote{Numerically, we update the pressure coefficient at a junction $l$ according to Algorithm~\ref{alg:networksimulation} whenever one of the incoming Lagrangian markers changes its value from $\overline{w}=a$ to $\overline{w}=\overline{w}_{l-1}$.} at each merging junction with $t_1^* < t_2^* < \dots t_{10}^{*} < T = 12$ in the free flow scenario. For the congested case, we observe less changes in the pressure constants. For roads $l=1, \dots, 5$, the times $t_l^*$ are identical to the free flow scenario. On road 5, the contact discontinuity interacts with the shock wave downstream. The traffic beyond the shock wave has zero velocity and the propagation of the Lagrangian marker is stopped, see Figure~\ref{img:Sequential_wovertime}. 
	\begin{table}[tbhp]
		{\footnotesize \caption{Adaption times $t_l^*$ and constants $d_l$}
			\label{tab:Sequential_adaptiontimes}
			\begin{center}
				\begin{tabular}{| c| c | c| c| c| c | c| c | c  |c | c |}\hline
					& $t_1^*$ &  $t_2^*$ & $t_3^*$ &  $t_4^*$ &  $t_5^*$ &  $t_6^*$ &  $t_7^*$ &  $t_8^*$ &  $t_9^*$ &  $t_{10}^*$ \\
					\hline
					free flow &  0 & 0.42 & 0.84 & 1.42 & 2.14 &  3.6 & 5.8 & 7.48 & 8.74 & 9.66 \\
					congested &  0 & 0.42 & 0.84 & 1.42 & 2.14 & $\infty$ & $\infty$ & $\infty$ & $\infty$ & $\infty$ \\ \hline
				\end{tabular}\\[5ex]
				
				\begin{tabular}{| c | c| c| c| c | c| c | c  |c | c |} \hline
					%\toprule
					$d_1$ &  $d_2$  & $d_3$  &  $d_4$  &  $d_5$  &  $d_6$  &  $d_7$  &  $d_8$  &  $d_9$  &  $d_{10}$  \\
					\hline
					%\midrule
					1.0800 & 1.0427 & 1.0241 & 1.0141 & 1.0084 &  1.0051 & 1.0032 & 1.0020 & 1.0012 & 1.0008 \\ \hline
					%\bottomrule
				\end{tabular}
			\end{center}
		}
	\end{table}\\
	The constants $d_l$ of Lemma~\ref{lem:Sequential} are also shown in Table~\ref{tab:Sequential_adaptiontimes}. We observe that $d_l$ approaches the value one already for $N=10$. In the congested case, the constants are identical, but  the change in the Lagrangian marker is not propagated across the junctions $l=6, \dots,10$. Hence,  the pressure law is identical to the initial pressure law on these roads.

	\subsubsection{Comparison to the LWR model}
	We may re-write the LWR model as a special case of the ARZ model. In the LWR model, the velocity function $V_i(\rho) = \vmax (1- \rho_i / \rhomax_i) = (w_{i,0} i-p_{i,0}(\rho)) $ is used as the equilibrium velocity for a fixed value $w_{i,0}$ and fixed pressure $p_{i,0}(\rho) =
	w_{i,0}/\rhomax_i \rho .$  Note that, the value  $w_{i,0}i$ does {\bf not} depend on time and space in contrast to the ARZ model.  
	The network structure induces differences in the solution of the LWR and the AP model, that are not only connected to the modeling of the pressure function after merges, but are due to the coupling conditions in general. Using the supply and demand formulation~\eqref{eq:ADPAR_Demand}-\eqref{eq:ADPAR_Supply},  we couple the 2-1-junction in the LWR model using the demand and supply functions~\eqref{eq:LWR_demandsupply} and use the demand and supply evaluations $D_1 = d_1(\rho_{1,0}), D_2 = d_2(\rho_{2,0}), S_3 = s_3(\rho_{3,0})$ instead of~\eqref{eq:ADPAR_demandsupply_21}. Since there are no contact discontinuities present in the solution to the LWR model, we  compute numerical results of the LWR model with a Godunov scheme using the same numerical grid $\Delta x = 1/100, \Delta t = \Delta x /4$. Figure~\ref{img:Sequential_wovertime} shows the Lagrangian marker on each road at $t=T$. For the AP model, we obtain, as expected, the homogenized Lagrangian marker on each road in the free flow scenario. In the congested scenario, the Lagrangian markers are given by $\overline{w_l}$ on roads $l=1,\dots, 4$ and $\overline{w_5}$ at the beginning of road $5$. On the remaining part, the Lagrangian marker is equal to $a$. In the LWR model, the perturbation of the Lagrangian marker on the incoming road does not affect the solution on the remaining part of the network.

	\begin{figure}[tbhp]
		\begin{subfigure}{0.45\textwidth}
			\input{./img/NumericalExamples/Sequential/zzz_w_freeandcongested}
			\subcaption{Lagrangian marker (roads 0-10)}
			\label{img:Sequential_wovertime}
		\end{subfigure}
		%\subfloat[Lagrangian marker (roads 0-10)]{\label{img:Sequential_wovertime} \input{./img/NumericalExamples/Sequential/zzz_w_freeandcongested}}
		\begin{subfigure}{0.45\textwidth}
			\input{./img/NumericalExamples/Sequential/zzz_densitynodeA}
			\subcaption{Density at node $A$}
			\label{img:Sequential_densityatA}
		\end{subfigure}
		%\subfloat[Density at node $A$]{\label{img:Sequential_densityatA} \input{./img/NumericalExamples/Sequential/zzz_densitynodeA}}
		\caption{Lagrangian marker and densities in the sequential network.}
	\end{figure}
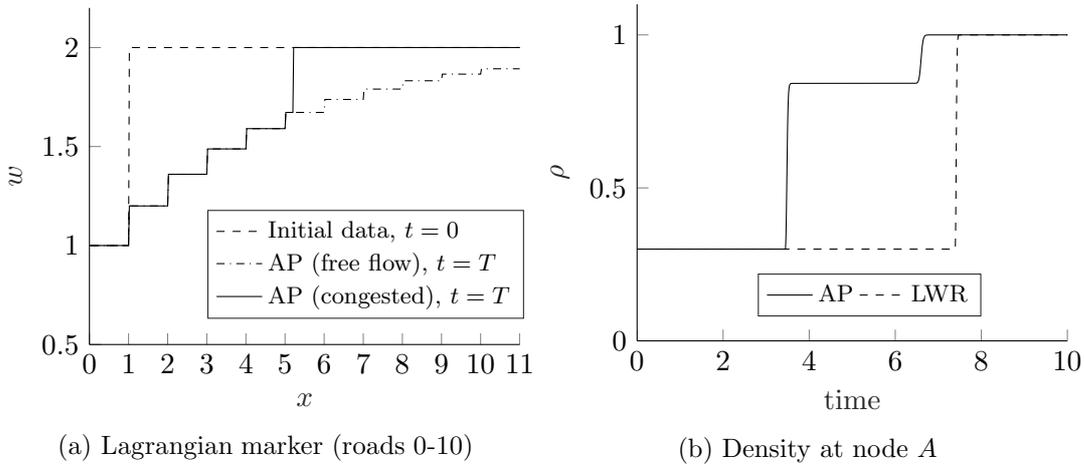

	%%%%%%%%%%%%%%%%%%%%%%%%%%%%%%%%%%%%%%%%%%%%%%%%%%%%%%%%%%%%%%%%%%%%%%%%%%%%%%%%%%%%%%%%%%%%%%%%%%%%%%%%%%%%%%%%%%%%%%%%%%%%%%%%%%%%%%%%%%%%%%%%%%%%%%%%%%%%%%%%%%%%%%%%%%%%%%%%%%%%%%%%%%%%%%%%%%%%%%%%%%%%%%%%%%%%%%%%%%%%%%%%%%%%%%%%%%%%%%%%%%%%%%%%%%%%%%%%%%%%%%%%%%%%%%%%%%%%%%%%%%%%%%%%%%%%%%%%%%%%%%%%%%%%%%%%%%%%%%%%%%%%%%%%%%%%%%%%%%%%%%%%%%%%%%%%%%%%%%%%%%%%%%%%%%%%%%%%%%%%%%%%%%%%%%%%%%%%%%%%%%%%%%%%%%%%%%%%%%%%%%%%%%%%%%%%%%%%%%%%%%%%%%%%%%%%%%%%%%%%%%%%%%%%%%%%%%%%%%%%%%%%%%%%%%%%%%%%%%%%%%%%%%%%%%%%%%%%%%%%%%%%%%%%%%%%%%%%%%%%%%%%%%%%%%%%%%%%%%%%%%%%%%%%%%%%%%%%%%%%%%%%%%%%%%%%%%%%%%%%%%%%%%%%%%%%%%%%%%%%%%%%%%%%%%%%%%%%%%%%%%%%%%%%%%%%%%%%%%%%%%%%%%%%%%%%%%%%%%%%%%%%%%%%%%%%%%%%%%%%%%%%%%%%%%%%%%%%%%%%%%%%%%%%%%%%
	
	\section{Conclusion}
	We have presented a suitable approximation of the homogenized pressure appearing as necessary coupling conditions for networked ARZ equations.  The novel pressure approximation 
	has been analyzed in detail and allows for an efficient numerical computation even for time--dependent boundary data and complex networks. A numerical method has been proposed to 
	efficiently compute coupled second--order traffic flow models on networks.  The numerical results demonstrate the performance 
	of the proposed procedure and highlight differences to the homogenized pressure and the LWR model predictions.
	
	\section*{Acknowledgment}
	S. G\"ottlich and J. Weissen were supported by the DFG grants GO 1920/7,10 and the DAAD projects 57444394 (USA), 57445223 (France)  
	while M.Herty gratefully acknowledges support through the DFG grants HE 5386/15,18,19, 320021702/GRK2326 and DFG EXC-2023 Internet of Production-390621612.
	The research of S. G\"ottlich and M. Herty was also supported by the joint BMBF grant ENets.
	
	\bibliography{Literature}
\end{document}

%% file: img/derivationofadaptedpressure/zzz_pressure_difference.tex
\begin{tikzpicture}
\setlength\fwidth{0.85\textwidth}
\begin{axis}[%
width=0.951\fwidth,
height=0.75\fwidth,
at={(0\fwidth,0\fwidth)},
scale only axis,
xmin=0,
xmax=4,
xlabel style={font=\color{white!15!black}},
xlabel={$\rho$},
ymin=0,
ymax=0.6,
ylabel style={font=\color{white!15!black}},
ylabel={pressure difference},
axis background/.style={fill=white},
axis x line*=bottom,
axis y line*=left,
legend style={legend cell align=left, align=left, draw=white!15!black}
]
\addplot [color=black, dashdotted]
  table[row sep=crcr]{%
0	0.5\\
0.032	0.484255934497534\\
0.0640000000000001	0.469022953566001\\
0.0960000000000001	0.454298715905187\\
0.128	0.440079358831523\\
0.16	0.426359556046886\\
0.192	0.413132595695855\\
0.224	0.400390476101966\\
0.257	0.387748244549075\\
0.29	0.375600614675012\\
0.323	0.363935295731072\\
0.357	0.352407006923057\\
0.391	0.341361481203485\\
0.426	0.33047861043467\\
0.461	0.320072656422384\\
0.497	0.309847785882599\\
0.534	0.299823605718747\\
0.571	0.290269268023225\\
0.609	0.280923137568032\\
0.648	0.27179862369764\\
0.688	0.262906912137274\\
0.729	0.254257020162195\\
0.772	0.24566130164828\\
0.816	0.237340220348926\\
0.861	0.229295612292171\\
0.908	0.221363605770994\\
0.956	0.213725379034194\\
1.006	0.206231273986138\\
1.058	0.19890177908836\\
1.112	0.191753970233526\\
1.168	0.184801664930559\\
1.227	0.177943143883375\\
1.288	0.171313436661017\\
1.352	0.164818648698993\\
1.419	0.158480558538035\\
1.489	0.152316731556677\\
1.563	0.146261957616285\\
1.641	0.140343509630991\\
1.723	0.13458345533896\\
1.81	0.128936651831242\\
1.902	0.123430546847957\\
1.999	0.118086797524022\\
2.102	0.112873274888637\\
2.212	0.107769335582343\\
2.329	0.102804324146335\\
2.454	0.0979637730896643\\
2.588	0.0932404261698832\\
2.731	0.0886630754492428\\
2.885	0.084197825373443\\
3.05	0.0798753845704034\\
3.228	0.0756733412112531\\
3.421	0.0715803152495686\\
3.63	0.0676112184941431\\
3.857	0.0637630975852561\\
4.001	0.0615378876247643\\
};
\addlegendentry{$\abs{p^*(\rho) - p_0(\rho)}$}

\addplot [color=black]
  table[row sep=crcr]{%
0	0.5\\
0.032	0.483747997989597\\
0.0640000000000001	0.468007080550128\\
0.0960000000000001	0.452774906381378\\
0.128	0.438047612799778\\
0.16	0.423819873507203\\
0.192	0.410084976648236\\
0.224	0.396834920546411\\
0.257	0.38366887946971\\
0.29	0.370997440071837\\
0.323	0.358808311604088\\
0.357	0.346740340256391\\
0.391	0.335155131997136\\
0.426	0.323716705672766\\
0.461	0.312755196104923\\
0.497	0.301958896993709\\
0.534	0.291347415242556\\
0.572	0.280937995770116\\
0.61	0.271000821870441\\
0.649	0.261269088038176\\
0.689	0.251753948180847\\
0.73	0.242464392538478\\
0.773	0.233197127803702\\
0.817	0.224188194926978\\
0.862	0.215439418746355\\
0.909	0.206771250977799\\
0.957	0.198380510052989\\
1.007	0.190101843846752\\
1.059	0.181955709467474\\
1.113	0.173959158311994\\
1.169	0.166125989984898\\
1.228	0.158338726512974\\
1.289	0.150748116969535\\
1.353	0.143244499949337\\
1.42	0.135849631074897\\
1.491	0.128480396493461\\
1.565	0.121263193856811\\
1.643	0.114118222522892\\
1.725	0.107067516698569\\
1.812	0.100050155469323\\
1.904	0.093093545354912\\
2.001	0.0862193199352292\\
2.104	0.0793795479899293\\
2.214	0.072537759315642\\
2.331	0.0657232650444479\\
2.456	0.0589057628634437\\
2.59	0.0520621489684805\\
2.733	0.0452211917999987\\
2.887	0.0383173133656998\\
3.053	0.0313404794018393\\
3.232	0.0242822858776295\\
3.425	0.0171349907021847\\
3.635	0.00982313890252318\\
3.863	0.00234976325053005\\
3.94	7.81207875295564e-05\\
4.001	0.0019700488831722\\
};
\addlegendentry{$\abs{p^*(\rho) - p^{**}(\rho)}$}

\end{axis}
\end{tikzpicture}%

%% file: img/derivationofadaptedpressure/zzz_fluxwithdifferentpressures.tex
\begin{tikzpicture}
\setlength\fwidth{0.85\textwidth}
\begin{axis}[%
width=0.951\fwidth,
height=0.75\fwidth,
at={(0\fwidth,0\fwidth)},
scale only axis,
xmin=0,
xmax=4,
xlabel style={font=\color{white!15!black}},
xlabel={$\rho$},
ymin=0,
ymax=4,
ylabel style={font=\color{white!15!black}},
ylabel={flux},
axis background/.style={fill=white},
axis x line*=bottom,
axis y line*=left,
legend style={at={(0.75,0.45)},legend cell align=left, align=left, draw=white!15!black}
]
\addplot [color=black, dashdotted]
  table[row sep=crcr]{%
0	0\\
0.069	0.271239\\
0.137	0.529231\\
0.203	0.770791\\
0.268	1.000176\\
0.332	1.217776\\
0.395	1.423975\\
0.457	1.619151\\
0.518	1.803676\\
0.578	1.977916\\
0.637	2.142231\\
0.695	2.296975\\
0.752	2.442496\\
0.807	2.576751\\
0.861	2.702679\\
0.914	2.820604\\
0.966	2.930844\\
1.017	3.033711\\
1.067	3.129511\\
1.116	3.218544\\
1.164	3.301104\\
1.211	3.377479\\
1.257	3.447951\\
1.302	3.512796\\
1.346	3.572284\\
1.389	3.626679\\
1.431	3.676239\\
1.473	3.722271\\
1.514	3.763804\\
1.554	3.801084\\
1.593	3.834351\\
1.631	3.863839\\
1.669	3.890439\\
1.706	3.913564\\
1.742	3.933436\\
1.778	3.950716\\
1.814	3.965404\\
1.849	3.977199\\
1.884	3.986544\\
1.919	3.993439\\
1.954	3.997884\\
1.988	3.999856\\
2.022	3.999516\\
2.056	3.996864\\
2.09	3.9919\\
2.125	3.984375\\
2.16	3.9744\\
2.195	3.961975\\
2.231	3.946639\\
2.267	3.928711\\
2.304	3.907584\\
2.341	3.883719\\
2.378	3.857116\\
2.416	3.826944\\
2.455	3.792975\\
2.494	3.755964\\
2.534	3.714844\\
2.575	3.669375\\
2.617	3.619311\\
2.66	3.5644\\
2.704	3.504384\\
2.748	3.440496\\
2.793	3.371151\\
2.839	3.296079\\
2.886	3.215004\\
2.934	3.127644\\
2.983	3.033711\\
3.033	2.932911\\
3.084	2.824944\\
3.136	2.709504\\
3.189	2.586279\\
3.243	2.454951\\
3.298	2.315196\\
3.354	2.166684\\
3.411	2.009079\\
3.47	1.8391\\
3.53	1.6591\\
3.591	1.468719\\
3.653	1.267591\\
3.716	1.055344\\
3.78	0.831599999999999\\
3.845	0.595975000000001\\
3.911	0.348079\\
3.978	0.0875160000000008\\
4.001	-0.00400100000000148\\
};
\addlegendentry{for $p_0(\rho)$}

\addplot [color=black]
  table[row sep=crcr]{%
0	0\\
0.0680000000000001	0.267302603174603\\
0.135	0.521485714285714\\
0.201	0.762957714285714\\
0.266	0.992120888888889\\
0.33	1.20937142857143\\
0.393	1.41509942857143\\
0.454	1.60661231746032\\
0.514	1.78761041269841\\
0.573	1.95845942857143\\
0.631	2.11951898412698\\
0.688	2.2711426031746\\
0.744	2.41367771428571\\
0.799	2.54746565079365\\
0.853	2.67284165079365\\
0.906	2.79013485714286\\
0.958	2.89966831746032\\
1.009	3.00175898412698\\
1.059	3.09671771428571\\
1.108	3.18484926984127\\
1.156	3.26645231746032\\
1.203	3.34181942857143\\
1.249	3.41123707936508\\
1.294	3.47498565079365\\
1.338	3.53333942857143\\
1.381	3.5865666031746\\
1.423	3.63492926984127\\
1.464	3.67868342857143\\
1.504	3.71807898412698\\
1.543	3.75335974603175\\
1.582	3.78555022222222\\
1.62	3.81394285714286\\
1.657	3.83876926984127\\
1.693	3.86025498412698\\
1.729	3.87910755555556\\
1.765	3.89532698412698\\
1.8	3.90857142857143\\
1.835	3.91932698412698\\
1.87	3.92759365079365\\
1.904	3.93324088888889\\
1.938	3.93653942857143\\
1.972	3.93748926984127\\
2.006	3.93609041269841\\
2.04	3.93234285714286\\
2.074	3.9262466031746\\
2.108	3.91780165079365\\
2.143	3.90665498412698\\
2.178	3.89301942857143\\
2.213	3.87689498412698\\
2.249	3.85771326984127\\
2.285	3.83589841269841\\
2.322	3.81073371428571\\
2.359	3.78278755555556\\
2.397	3.75119085714286\\
2.436	3.715712\\
2.475	3.67714285714286\\
2.515	3.6343746031746\\
2.556	3.58716342857143\\
2.598	3.53525942857143\\
2.641	3.4784066031746\\
2.685	3.41634285714286\\
2.73	3.3488\\
2.776	3.27550374603175\\
2.823	3.19617371428571\\
2.87	3.11235555555556\\
2.918	3.02212165079365\\
2.967	2.92517942857143\\
3.017	2.82123022222222\\
3.068	2.70996926984127\\
3.12	2.59108571428571\\
3.174	2.46181485714286\\
3.229	2.32405993650794\\
3.285	2.17748571428571\\
3.342	2.02175085714286\\
3.4	1.85650793650794\\
3.459	1.68140342857143\\
3.519	1.49607771428571\\
3.58	1.30016507936508\\
3.642	1.09329371428572\\
3.705	0.875085714285715\\
3.769	0.64515707936508\\
3.834	0.403117714285715\\
3.9	0.14857142857143\\
3.938	-0.00200025396825154\\
};
\addlegendentry{for $p^{**}(\rho)$}

\addplot [color=black, dashed, line width=2.0pt]
  table[row sep=crcr]{%
0	0\\
0.089	0.307363605378292\\
0.172	0.585945159890363\\
0.25	0.839902949199448\\
0.324	1.07322110140586\\
0.395	1.289638258993\\
0.463	1.48970617629691\\
0.529	1.67685049685985\\
0.593	1.85146329423953\\
0.655	2.0139879946922\\
0.715	2.16490598712801\\
0.773	2.30472503290615\\
0.83	2.43622560550259\\
0.885	2.55748429290407\\
0.939	2.67109894967708\\
0.992	2.77731823865745\\
1.044	2.87639690869825\\
1.094	2.96683351697706\\
1.143	3.05083857982651\\
1.191	3.12866968964668\\
1.238	3.20058388892025\\
1.284	3.26683677299815\\
1.329	3.32768174958596\\
1.373	3.38336942697487\\
1.416	3.43414710789359\\
1.458	3.48025836996886\\
1.499	3.52194271722474\\
1.539	3.5594352899021\\
1.579	3.59378710336052\\
1.618	3.62425194682562\\
1.656	3.65105620912788\\
1.693	3.67442121030014\\
1.73	3.695086056928\\
1.767	3.71304826401274\\
1.803	3.72792876316409\\
1.839	3.74024661177534\\
1.874	3.74976371586167\\
1.909	3.75685530658175\\
1.944	3.76151994849669\\
1.979	3.76375630581534\\
2.013	3.76360238039426\\
2.047	3.76115470910573\\
2.082	3.75623806674347\\
2.117	3.74888870177584\\
2.152	3.73910566793718\\
2.187	3.72688808081297\\
2.223	3.71178066399932\\
2.259	3.69409587259581\\
2.296	3.67323325023262\\
2.333	3.64964644760068\\
2.37	3.6233347184787\\
2.408	3.59347473720597\\
2.447	3.55983879943954\\
2.486	3.5231731533954\\
2.526	3.4824194329051\\
2.567	3.43733785810274\\
2.609	3.38768261771162\\
2.652	3.33320187065935\\
2.695	3.27503435227041\\
2.739	3.21169706750729\\
2.784	3.14292612299593\\
2.83	3.06845160197577\\
2.877	2.98799756586693\\
2.925	2.90128205582665\\
2.974	2.80801709428289\\
3.024	2.70790868643496\\
3.075	2.60065682171225\\
3.127	2.48595547518398\\
3.18	2.36349260891408\\
3.234	2.23295017325664\\
3.289	2.09400410808855\\
3.345	1.94632434397702\\
3.402	1.78957480328056\\
3.46	1.6234134011828\\
3.519	1.44749204665932\\
3.58	1.2582951161061\\
3.642	1.05837827603446\\
3.705	0.847369405070346\\
3.769	0.624890379767551\\
3.834	0.3905570752364\\
3.9	0.143979365713015\\
3.938	-0.00193871136279533\\
};
\addlegendentry{for $p^{*}(\rho)$}

\end{axis}
\end{tikzpicture}%

%% file: img/NumericalExamples/zzz_Riemann21_HerMouRas_comparison.tex
\begin{tikzpicture}
\setlength\fwidth{0.85\textwidth}
\begin{axis}[%
width=0.951\fwidth,
height=0.75\fwidth,
at={(0\fwidth,0\fwidth)},
scale only axis,
xmin=0,
xmax=0.5,
xlabel style={font=\color{white!15!black}},
xlabel={$x$},
ymin=1,
ymax=4,
ylabel style={font=\color{white!15!black}},
ylabel={$\rho(x,t)$},
axis background/.style={fill=white},
axis x line*=bottom,
axis y line*=left,
legend style={at={(0.7,1.05)}, legend cell align=left, align=left, draw=white!15!black,legend columns=1,font=\footnotesize}
]
\addplot [color=black, line width=1.0pt]
  table[row sep=crcr]{%
0.0012500000000002	1.99489796054904\\
0.0687500000000001	1.71938775512687\\
0.07125	1.71428571289552\\
0.27875	1.71428571289552\\
0.28125	3\\
0.50125	3\\
};
\addlegendentry{solution for $p^{**}$ };

\addplot [color=black, dotted, line width=1.0pt]
  table[row sep=crcr]{%
0.0012500000000002	2.02605641906086\\
0.0562499999999999	1.79329399197352\\
0.0987499999999999	1.61227464620003\\
0.11125	1.55877031096825\\
0.11375	1.55555555555576\\
0.27875	1.55555555555576\\
0.28125	3\\
0.50125	3\\
};
\addlegendentry{solution for $p^{*}$ };

\end{axis}
\end{tikzpicture}%

%% file: img/NumericalExamples/zzz_Riemann21_convergence.tex
\definecolor{mycolor1}{rgb}{0.00000,0.44700,0.74100}%
\begin{tikzpicture}
\setlength\fwidth{0.85\textwidth}
\begin{axis}[%
width=0.951\fwidth,
height=0.75\fwidth,
at={(0\fwidth,0\fwidth)},
scale only axis,
xmin=0,
xmax=0.5,
xlabel style={font=\color{white!15!black}},
xlabel={$x$},
ymin=1,
ymax=9,
ylabel style={font=\color{white!15!black}},
ylabel={$\rho(x,t)$},
axis background/.style={fill=white},
axis x line*=bottom,
axis y line*=left,
legend style={at={(0.55,1.05)}, anchor=north, legend cell align=left, align=left, draw=white!15!black, legend columns=2,font=\footnotesize}
]
\addplot [color=black, line width=1.0pt]
  table[row sep=crcr]{%
0.0012500000000002	1.99489796054904\\
0.0687500000000001	1.71938775512687\\
0.07125	1.71428571289552\\
0.27875	1.71428571289552\\
0.28125	3\\
0.50125	3\\
};
\addlegendentry{sol. for $p^{**}$}

\addplot [color=blue, dotted, line width=1.0pt]
  table[row sep=crcr]{%
0.00499999999999989	1.93588708310867\\
0.0150000000000001	1.89799969053198\\
0.0249999999999999	1.86456134401663\\
0.0350000000000001	1.83443805438881\\
0.0449999999999999	1.80753542333553\\
0.0550000000000002	1.7840448097683\\
0.0649999999999999	1.7642151145911\\
0.0750000000000002	1.7482106700442\\
0.085	1.7359987185979\\
0.0950000000000002	1.72728408624638\\
0.105	1.72152260784112\\
0.115	1.71801863140759\\
0.125	1.71606642211358\\
0.135	1.71507125916847\\
0.145	1.71460657640464\\
0.165	1.71432853955161\\
0.225	1.7142857327461\\
0.285	1.7142857142858\\
0.295	3\\
0.505	3\\
};
\addlegendentry{$\Delta x = 1/100$}

\addplot [color=blue, dashed, line width=1.0pt]
  table[row sep=crcr]{%
0.00249999999999995	1.96416671991074\\
0.00749999999999984	1.94228439556002\\
0.0125000000000002	1.92217236929286\\
0.0175000000000001	1.90303416778777\\
0.0225	1.88463055399314\\
0.0274999999999999	1.86688404539125\\
0.0325000000000002	1.8497881438417\\
0.0375000000000001	1.83337697059108\\
0.0425	1.81771263606463\\
0.0474999999999999	1.80287860511933\\
0.0525000000000002	1.78897463087877\\
0.0575000000000001	1.77611101612366\\
0.0625	1.76440080855455\\
0.0674999999999999	1.75394915558454\\
0.0725000000000002	1.74483989047231\\
0.0775000000000001	1.73712065334414\\
0.0825	1.73078929215407\\
0.0874999999999999	1.72578535213137\\
0.0924999999999998	1.72199030181802\\
0.0975000000000001	1.71923819339652\\
0.1025	1.71733511204511\\
0.1075	1.71608255390948\\
0.1125	1.71529858302311\\
0.1175	1.71483198654774\\
0.1275	1.71442521744737\\
0.1425	1.71429889199175\\
0.2425	1.71428571428572\\
0.2825	1.71428571428572\\
0.2875	3\\
0.5025	3\\
};
\addlegendentry{$\Delta x = 1/200$}

\addplot [color=blue, line width=1.0pt]
  table[row sep=crcr]{%
0.0012500000000002	1.98094205723231\\
0.00375000000000014	1.96919767230349\\
0.00875000000000004	1.94778929282359\\
0.0162499999999999	1.91752896371332\\
0.0237500000000002	1.88844922923711\\
0.03125	1.86035552217728\\
0.0387499999999998	1.83335756288082\\
0.0437500000000002	1.81610639734472\\
0.0487500000000001	1.7996032277179\\
0.05375	1.78402145061584\\
0.0587499999999999	1.76957538474674\\
0.0637500000000002	1.75651343586688\\
0.0662500000000001	1.75058362238953\\
0.0687500000000001	1.74509680359237\\
0.07125	1.74008073975885\\
0.07375	1.73555757829012\\
0.0762499999999999	1.73154163279487\\
0.0787499999999999	1.728037319631\\
0.0812499999999998	1.72503757801987\\
0.0837500000000002	1.72252311186071\\
0.0862500000000002	1.72046272109128\\
0.0887500000000001	1.71881483118386\\
0.0912500000000001	1.71753010753082\\
0.09375	1.71655481569412\\
0.0962499999999999	1.71583443057973\\
0.0987499999999999	1.71531696173102\\
0.10375	1.71471015146004\\
0.10875	1.71444400286851\\
0.11875	1.71430223529651\\
0.16875	1.71428571428702\\
0.28125	1.71428571428572\\
0.28375	3\\
0.50125	3\\
};
\addlegendentry{$\Delta x = 1/400$}

\draw[draw=black] (axis cs:0,1.71428571289552) rectangle (axis cs:0.09,1.99489796054904);
\addplot [color=black]
  table[row sep=crcr]{%
0	1.99489796054904\\
0.0451612903225804	6.79141104294478\\
};

\addplot [color=black]
  table[row sep=crcr]{%
0.0899999999999999	1.71428571289552\\
0.174193548387097	2.86503067484663\\
};
%\addlegendentry{data2}

\draw[draw=black] (axis cs:0.27,1.71428571289552) rectangle (axis cs:0.3,3);
\addplot [color=black]
  table[row sep=crcr]{%
0.27	3\\
0.354838709677419	7.28220858895705\\
};
%\addlegendentry{data3}

\addplot [color=black]
  table[row sep=crcr]{%
0.3	1.71428571289552\\
0.483870967741935	3.3558282208589\\
};
%\addlegendentry{data4}

\end{axis}

\begin{axis}[%
ticks=none,
width=0.245\fwidth,
height=0.368\fwidth,
at={(0.086\fwidth,0.175\fwidth)},
scale only axis,
xmin=0,
xmax=0.09,
ymin=1.71428571289552,
ymax=1.99489796054904,
axis background/.style={fill=white},
legend style={legend cell align=left, align=left, draw=white!15!black}
]
\addplot [color=mycolor1]
  table[row sep=crcr]{%
0.00124999999999997	1.99489796054904\\
0.0687500000000001	1.71938775512687\\
0.07125	1.71428571289552\\
0.0912500000000001	1.71428571289552\\
};
%\addlegendentry{data1}

\addplot [color=black, line width=1.0pt]
  table[row sep=crcr]{%
0.00124999999999997	1.99489796054904\\
0.0687500000000001	1.71938775512687\\
0.07125	1.71428571289552\\
0.0912500000000001	1.71428571289552\\
};
%\addlegendentry{data2}

\addplot [color=blue, dotted, line width=1.0pt]
  table[row sep=crcr]{%
0.00499999999999989	1.93588708310867\\
0.0149999999999999	1.89799969053198\\
0.0249999999999999	1.86456134401663\\
0.0349999999999999	1.83443805438881\\
0.0449999999999999	1.80753542333553\\
0.0549999999999999	1.7840448097683\\
0.0649999999999999	1.7642151145911\\
0.075	1.7482106700442\\
0.085	1.7359987185979\\
0.095	1.72728408624638\\
};
%\addlegendentry{data3}

\addplot [color=blue, dashed, line width=1.0pt]
  table[row sep=crcr]{%
0.00249999999999995	1.96416671991074\\
0.00750000000000006	1.94228439556002\\
0.0125	1.92217236929286\\
0.0175000000000001	1.90303416778777\\
0.0225	1.88463055399314\\
0.0275000000000001	1.86688404539125\\
0.0325	1.8497881438417\\
0.0375000000000001	1.83337697059108\\
0.0425	1.81771263606463\\
0.0475000000000001	1.80287860511933\\
0.0525	1.78897463087877\\
0.0575000000000001	1.77611101612366\\
0.0625	1.76440080855455\\
0.0675000000000001	1.75394915558454\\
0.0725	1.74483989047231\\
0.0774999999999999	1.73712065334414\\
0.0825	1.73078929215407\\
0.0874999999999999	1.72578535213137\\
0.0925	1.72199030181802\\
};
%\addlegendentry{data4}

\addplot [color=blue, line width=1.0pt]
  table[row sep=crcr]{%
0.00124999999999997	1.98094205723231\\
0.00374999999999992	1.96919767230349\\
0.00625000000000009	1.95829390478917\\
0.00875000000000004	1.94778929282359\\
0.01125	1.93753388478475\\
0.0137499999999999	1.9274590905061\\
0.0162500000000001	1.91752896371332\\
0.01875	1.90772358102131\\
0.02125	1.89803206130447\\
0.0237499999999999	1.88844922923711\\
0.0262500000000001	1.8789738843588\\
0.0287500000000001	1.86960786209606\\
0.03125	1.86035552217728\\
0.0337499999999999	1.85122348575211\\
0.0362500000000001	1.84222052667591\\
0.0387500000000001	1.83335756288082\\
0.04125	1.82464771360197\\
0.04375	1.81610639734472\\
0.0462499999999999	1.80775144832036\\
0.0487500000000001	1.7996032277179\\
0.05125	1.79168470165197\\
0.05375	1.78402145061584\\
0.0562499999999999	1.77664156657166\\
0.0587500000000001	1.76957538474674\\
0.06125	1.76285499001112\\
0.06375	1.75651343586688\\
0.0662499999999999	1.75058362238953\\
0.0687500000000001	1.74509680359237\\
0.07125	1.74008073975885\\
0.07375	1.73555757829012\\
0.0762499999999999	1.73154163279487\\
0.0787500000000001	1.728037319631\\
0.08125	1.72503757801987\\
0.08375	1.72252311186071\\
0.0862499999999999	1.72046272109128\\
0.0887500000000001	1.71881483118386\\
0.0912500000000001	1.71753010753082\\
};
%\addlegendentry{data5}

\end{axis}

\begin{axis}[%
ticks = none,
width=0.245\fwidth,
height=0.368\fwidth,
at={(0.675\fwidth,0.221\fwidth)},
scale only axis,
xmin=0.27,
xmax=0.3,
ymin=1.71428571289552,
ymax=3,
axis background/.style={fill=white},
legend style={legend cell align=left, align=left, draw=white!15!black}
]
\addplot [color=mycolor1]
  table[row sep=crcr]{%
0.26875	1.71428571289552\\
0.27875	1.71428571289552\\
0.28125	3\\
0.30125	3\\
};
%\addlegendentry{data1}

\addplot [color=black, line width=1.0pt]
  table[row sep=crcr]{%
0.26875	1.71428571289552\\
0.27875	1.71428571289552\\
0.28125	3\\
0.30125	3\\
};
%\addlegendentry{data2}

\addplot [color=blue, dotted, line width=1.0pt]
  table[row sep=crcr]{%
0.267	1.71428571430126\\
0.285	1.7142857142858\\
0.295	3\\
0.303	3\\
};
%\addlegendentry{data3}

\addplot [color=blue, dashed, line width=1.0pt]
  table[row sep=crcr]{%
0.2675	1.71428571428571\\
0.2825	1.71428571428572\\
0.2875	3\\
0.3025	3\\
};
%\addlegendentry{data4}

\addplot [color=blue, line width=1.0pt]
  table[row sep=crcr]{%
0.26875	1.71428571428572\\
0.28125	1.71428571428572\\
0.28375	3\\
0.30125	3\\
};
%\addlegendentry{data5}

\end{axis}
\end{tikzpicture}%

%% file: img/NumericalExamples/Sequential/carpath_1shock.tex
\begin{tikzpicture}
\setlength\fwidth{0.8\textwidth}
\begin{axis}[%
width=0.951\fwidth,
height=0.75\fwidth,
at={(0\fwidth,0\fwidth)},
scale only axis,
xmin=0.5,
xmax=2.85,
xtick={  1, 2},
xticklabels={$M_1$,  $M_2$},
xlabel style={font=\color{white!15!black}},
ymin=0,
ymax=1,
ytick={0,0.85},
yticklabels={$t^s$,  $t^{s+1}$},
ylabel style={font=\color{white!15!black}},
axis background/.style={fill=white},
axis x line*=bottom,
axis y line*=left,
legend style={legend cell align=left, align=left, draw=white!15!black}
]
\addplot [color=black, dashed]
table[row sep=crcr]{%
	1	0\\
	1	1\\
};
%\addlegendentry{data1}

\addplot [color=black, dashed]
table[row sep=crcr]{%
	2	0\\
	2	1\\
};
%\addlegendentry{data2}

\addplot [color=black]
  table[row sep=crcr]{%
1	0\\
1.5	0.5\\
};
%\addlegendentry{data3}

\addplot [color=black, dotted]
  table[row sep=crcr]{%
2	0\\
1.5	0.5\\
};
%\addlegendentry{data4}

\addplot [color=black]
table[row sep=crcr]{%
	1.5	0.5\\
	1.7	1\\
};
%\addlegendentry{data5}

\addplot [color=black, dotted]
table[row sep=crcr]{%
	1	0\\
	0.5	0.5\\
};
%\addlegendentry{data6}

\node (A) at (axis cs: 0.8, 0.2) {$w_0$};
\node (B) at (axis cs: 1.8, 0.2) {$w_1$};
\node (C) at (axis cs: 1.2, 0.5) {$\bar{w_1}$};
\node (E) at (axis cs: 2.7, 0.2) {$w_2$};
\end{axis}
\end{tikzpicture}%

%% file: img/NumericalExamples/Sequential/carpath_1rarefaction.tex
\begin{tikzpicture}
\setlength\fwidth{0.8\textwidth}
\begin{axis}[%
width=0.951\fwidth,
height=0.75\fwidth,
at={(0\fwidth,0\fwidth)},
scale only axis,
xmin=0.5,
xmax=2.85,
xtick={  1, 2},
xticklabels={$M_1$,  $M_2$},
xlabel style={font=\color{white!15!black}},
ymin=0,
ymax=1,
ytick={0,0.85},
yticklabels={$t^s$,  $t^{s+1}$},
ylabel style={font=\color{white!15!black}},
axis background/.style={fill=white},
axis x line*=bottom,
axis y line*=left,
legend style={legend cell align=left, align=left, draw=white!15!black}
]

\addplot [color=black, dashed]
table[row sep=crcr]{%
	1	0\\
	1	1\\
};
%\addlegendentry{data1}

\addplot [color=black, dashed]
table[row sep=crcr]{%
	2	0\\
	2	1\\
};
%\addlegendentry{data2}

\addplot [color=black]
  table[row sep=crcr]{%
1	0\\
2	0.2\\
};
%\addlegendentry{data3}

\addplot [color=black]
  table[row sep=crcr]{%
2	0.2\\
2.2	0.5\\
};
%\addlegendentry{data4}

\addplot [color=black]
table[row sep=crcr]{%
	2.2	0.5\\
	2.4 0.667\\
};
%\addlegendentry{data5}

\addplot [color=black]
table[row sep=crcr]{%
	2.4 0.667\\
	2.55 0.6875\\
};
%\addlegendentry{data6}

\addplot [color=black]
table[row sep=crcr]{%
	2.55 0.6875\\
	2.8  0.69\\
};
%\addlegendentry{data7}

\addplot [color=black, dotted]
table[row sep=crcr]{%
	2	0\\
	2.2	0.5\\
};
%\addlegendentry{data8}

\addplot [color=black, dotted]
table[row sep=crcr]{%
	2	0\\
	2.4	0.667\\
};
%\addlegendentry{data9}

\addplot [color=black, dotted]
table[row sep=crcr]{%
	2	0\\
	2.55	0.6875\\
};
%\addlegendentry{data10}

\addplot [color=black, dotted]
table[row sep=crcr]{%
	1	0\\
	1.2	1\\
};
%\addlegendentry{data8}

\addplot [color=black, dotted]
table[row sep=crcr]{%
	1	0\\
	1.4	1\\
};
%\addlegendentry{data9}

\addplot [color=black, dotted]
table[row sep=crcr]{%
	1	0\\
	1.55	1\\
};
%\addlegendentry{data10}

\node (A) at (axis cs: 0.8, 0.2) {$w_0$};
\node (B) at (axis cs: 1.8, 0.08) {$w_1$};
\node (C) at (axis cs: 1.2, 0.5) {$\bar{w_1}$};
\node (D) at (axis cs: 2.2, 0.8) {$\bar{w_2}$};
\node (E) at (axis cs: 2.7, 0.2) {$w_2$};
\end{axis}
\end{tikzpicture}%

%% file: img/NumericalExamples/Sequential/disturbance_xtdiagram.tex
\begin{tikzpicture}
\setlength\fwidth{0.4\textwidth}
\begin{axis}[%
width=0.951\fwidth,
height=0.75\fwidth,
at={(0\fwidth,0\fwidth)},
scale only axis,
xmin=0.5,
xmax=3.3,
xtick={  1, 2, 3},
xticklabels={$M_1$,  $M_2$, $M3$},
xlabel style={font=\color{white!15!black}},
ymin=0,
ymax=1,
ytick={0,0.85},
yticklabels={0,  $t$},
ylabel style={font=\color{white!15!black}},
axis background/.style={fill=white},
axis x line*=bottom,
axis y line*=left,
legend style={legend cell align=left, align=left, draw=white!15!black}
]

\addplot [color=black, dashed]
table[row sep=crcr]{%
	1	0\\
	1	1\\
};
%\addlegendentry{data1}

\addplot [color=black, dashed]
table[row sep=crcr]{%
	2	0\\
	2	1\\
};
%\addlegendentry{data2}

\addplot [color=black, dashed]
table[row sep=crcr]{%
	3	0\\
	3	1\\
};
%\addlegendentry{data2a}

\addplot [color=black]
  table[row sep=crcr]{%
1	0\\
2	0.2\\
};
%\addlegendentry{data3}

\addplot [color=black]
  table[row sep=crcr]{%
2	0.2\\
2.2	0.5\\
};
%\addlegendentry{data4}

\addplot [color=black]
table[row sep=crcr]{%
	2.2	0.5\\
	2.4 0.667\\
};
%\addlegendentry{data5}

\addplot [color=black]
table[row sep=crcr]{%
	2.4 0.667\\
	2.55 0.6875\\
};
%\addlegendentry{data6}

\addplot [color=black]
table[row sep=crcr]{%
	2.55 0.6875\\
	2.8  0.69\\
};
%\addlegendentry{data7}

\addplot [color=black]
table[row sep=crcr]{%
	2.8  0.69\\
	2.8   1 \\
};
%\addlegendentry{data7a}

\addplot [color=black, dotted]
table[row sep=crcr]{%
	1	0\\
	0.5	0.5\\
};
%\addlegendentry{data7c}

\addplot [color=black, dotted]
table[row sep=crcr]{%
	2	0\\
	2.2	0.5\\
};
%\addlegendentry{data8}

\addplot [color=black, dotted]
table[row sep=crcr]{%
	2	0\\
	2.4	0.667\\
};
%\addlegendentry{data9}

\addplot [color=black, dotted]
table[row sep=crcr]{%
	2	0\\
	2.55	0.6875\\
};
%\addlegendentry{data10}

\addplot [color=black, dotted]
table[row sep=crcr]{%
	2.8	0.69\\
	3	0\\
};
%\addlegendentry{data10}

\node (A) at (axis cs: 0.8, 0.2) {$b$};
\node (B) at (axis cs: 1.8, 0.08) {$a$};
\node (C) at (axis cs: 1.2, 0.5) {$\bar{w_1}$};
\node (D) at (axis cs: 2.2, 0.8) {$\bar{w_2}$};
\node (E) at (axis cs: 2.7, 0.2) {$a$};
\node (F) at (axis cs: 3.2, 0.2) {$a$};
\end{axis}
\end{tikzpicture}%

%% file: img/NumericalExamples/Sequential/zzz_w_freeandcongested.tex
\begin{tikzpicture}
\setlength\fwidth{0.85\textwidth}
\begin{axis}[%
width=0.951\fwidth,
height=0.75\fwidth,
at={(0\fwidth,0\fwidth)},
scale only axis,
xmin=0,
xmax=11,
xtick={ 0,  1,  2,  3,  4,  5,  6,  7,  8,  9, 10, 11},
xlabel style={font=\color{white!15!black}},
xlabel={$x$},
ymin=0.5,
ymax=2.2,
ylabel style={font=\color{white!15!black}},
ylabel={$w$},
axis background/.style={fill=white},
axis x line*=bottom,
axis y line*=left,
legend style={at={(1.01,0.4)},legend cell align=left, align=left, draw=white!15!black, font=\footnotesize}
]
\addplot [color=black, dashed]
  table[row sep=crcr]{%
0.0199999999999996	1\\
1	1\\
1.02	2\\
11	2\\
};
\addlegendentry{Initial data, $t=0$ };

\addplot [color=black, dashdotted]
  table[row sep=crcr]{%
0.0199999999999996	1\\
1	1\\
1.02	1.2\\
2	1.2\\
2.02	1.36\\
3	1.36\\
3.02	1.488\\
4	1.488\\
4.02	1.5904\\
5	1.5904\\
5.02	1.67232\\
6	1.67232\\
6.02	1.737856\\
7	1.737856\\
7.02	1.7902848\\
8	1.7902848\\
8.02	1.83222784\\
9	1.83222784\\
9.02	1.865782272\\
10	1.865782272\\
10.02	1.8926258176\\
11	1.8926258176\\
};
\addlegendentry{AP (free flow), $t=T$};

\addplot [color=black]
  table[row sep=crcr]{%
0.0199999999999996	1\\
1	1\\
1.02	1.2\\
2	1.2\\
2.02	1.36\\
3	1.36\\
3.02	1.488\\
4	1.488\\
4.02	1.5904\\
5	1.5904\\
5.02	1.67232\\
5.2	1.67232\\
5.22	2\\
11	2\\
};

\addlegendentry{AP (congested), $t=T$};

\end{axis}
\end{tikzpicture}%

%% file: img/NumericalExamples/Sequential/zzz_densitynodeA.tex
\begin{tikzpicture}
\setlength\fwidth{0.85\textwidth}
\begin{axis}[%
width=0.951\fwidth,
height=0.75\fwidth,
at={(0\fwidth,0\fwidth)},
scale only axis,
xmin=0,
xmax=10,
xlabel style={font=\color{white!15!black}},
xlabel={time},
ymin=0,
ymax=1.1,
ylabel style={font=\color{white!15!black}},
ylabel={$\rho$},
axis background/.style={fill=white},
axis x line*=bottom,
axis y line*=left,
legend style={at={(0.8,0.2)}, legend cell align=left, align=left, draw=white!15!black, legend columns=2, font=\footnotesize}
]
\addplot [color=black]
  table[row sep=crcr]{%
0	0.300000000000001\\
3.445	0.300000000000001\\
3.4475	0.301812638001405\\
3.45	0.305750693566207\\
3.4525	0.311927284874509\\
3.455	0.320421232425058\\
3.4575	0.331277452946544\\
3.46	0.344505320702117\\
3.465	0.37791223717425\\
3.47	0.419835819617919\\
3.4775	0.494772364994649\\
3.4925	0.651721439673196\\
3.4975	0.695278168468892\\
3.5025	0.731605228370571\\
3.5075	0.760597351739523\\
3.5125	0.782915351065464\\
3.5175	0.799613824017367\\
3.5225	0.811839500944405\\
3.5275	0.820647162646766\\
3.5325	0.826918183939524\\
3.5375	0.83134553165698\\
3.5425	0.834452502663183\\
3.5475	0.836623636201585\\
3.555	0.83870989570234\\
3.5625	0.83991809159269\\
3.575	0.840908037370815\\
3.595	0.841414374145943\\
3.65	0.841562411635875\\
6.48	0.84196599946465\\
6.5	0.842813695898961\\
6.5125	0.843999060847215\\
6.5225	0.845622962268511\\
6.5325	0.848189773415303\\
6.54	0.850992080576374\\
6.5475	0.854791704158488\\
6.555	0.859829652646072\\
6.5625	0.866331719883611\\
6.57	0.87445647874805\\
6.5775	0.884230104788793\\
6.585	0.895487076483652\\
6.62	0.952982654948341\\
6.6275	0.962878193686068\\
6.635	0.971222921787579\\
6.6425	0.978038767429263\\
6.65	0.983460386112117\\
6.6575	0.987680787955279\\
6.665	0.990909186852981\\
6.6725	0.993344136837534\\
6.6825	0.995653160983071\\
6.695	0.997483601896349\\
6.71	0.998716750871282\\
6.7325	0.999546763280007\\
6.7725	0.999933488819476\\
6.93	0.999998208173327\\
10.0025	0.999998219402842\\
};
\addlegendentry{AP}

\addplot [color=black, dashed]
  table[row sep=crcr]{%
0	0.300000000000001\\
7.3925	0.300000000000001\\
7.395	0.304588339348886\\
7.3975	0.315657092289232\\
7.4	0.333940172631301\\
7.4025	0.359747171231955\\
7.405	0.392942879461271\\
7.41	0.478716464168372\\
7.4275	0.821932312800309\\
7.4325	0.885837288669292\\
7.435	0.91000769278256\\
7.4375	0.929638117200854\\
7.44	0.945352187623955\\
7.4425	0.957783830468815\\
7.445	0.967526086848235\\
7.4475	0.97510368361997\\
7.45	0.980962942961494\\
7.4525	0.985472739012241\\
7.455	0.988931515447746\\
7.4575	0.991576921428274\\
7.4625	0.99513439720609\\
7.4675	0.997195351919382\\
7.4725	0.998385138475381\\
7.48	0.999294867197694\\
7.4925	0.999822021533213\\
7.525	0.999993084283842\\
10.0025	0.999997852533166\\
};
\addlegendentry{LWR}

\end{axis}
\end{tikzpicture}%